\documentclass[a4paper]{article}

\usepackage[plainpages=false,colorlinks=true,
            linkcolor=black,urlcolor=black,citecolor=black]{hyperref}
\usepackage{geometry}

\usepackage{amsmath,amssymb,amsthm, amsfonts,mathrsfs,cite}
\usepackage{color}
\usepackage{booktabs}
\usepackage{tikz}
\usepackage{scrextend}
\usepackage{arydshln}
\usepackage[font=scriptsize,labelfont=bf]{caption}

\usepackage{algpseudocode}
\usepackage{graphicx}
\usepackage{array, tabularx}
\usepackage{algorithm,algcompatible}
\usepackage{booktabs} 
\usepackage{paralist} 
\usepackage{verbatim} 
\usepackage{subfig} 
\usepackage{xcolor,marginnote,enumitem}
\usepackage[numbers]{natbib}
\numberwithin{equation}{section}

\theoremstyle{theorem}
\newtheorem{lemma}{Lemma}
\newtheorem{theorem}{Theorem}
\newtheorem{proposition}{Proposition}

\newtheorem{assumption}{Assumption}

\theoremstyle{remark}
\newtheorem{remark}{Remark}

\theoremstyle{definition}
\newtheorem{definition}{Definition}

\DeclareMathOperator{\dom}{dom}

\DeclareMathOperator*{\argmin}{arg\,min}

\DeclareMathOperator{\Div}{div}

\newcommand{\ZZ}{\mathbf{Z}}
\newcommand{\grad}{\nabla}

\newcommand{\lookUp}[1]{}

\newcommand{\norm}[2][]{\|{#2}\|_{#1}}

\newcommand{\scp}[3][]{\langle{#2},{#3}\rangle_{#1}}

\newcommand{\RR}{\mathbf{R}}

\newcommand{\mF}{\mathcal{F}}
\newcommand{\mG}{\mathcal{G}}

\newcommand{\mI}{\mathcal{I}}

\newcommand{\mP}{\mathcal{P}}
\newcommand{\mV}{\mathcal{V}}
\newcommand{\mN}{\mathcal{N}}

\algnewcommand\INPUT{\item[\textbf{Input:}]}%
\algnewcommand\OUTPUT{\item[\textbf{Output:}]}%
\newcommand{\gap}{\mathfrak{G}}

\DeclareMathAlphabet{\mathbfit}{OML}{cmm}{b}{it}

\date{}

\begin{document}

\title{An Investigation on Semismooth Newton based Augmented Lagrangian Method for Image Restoration}

\author{Hongpeng Sun\thanks{Institute for Mathematical Sciences,
		Renmin University of China, 100872 Beijing, China.
		Email: \href{mailto:hpsun@amss.ac.cn}{hpsun@amss.ac.cn}.} }

\maketitle
\begin{abstract}
Augmented Lagrangian method (also called as method of multipliers) is an important and powerful optimization method for lots of smooth or nonsmooth variational problems in modern signal processing, imaging, optimal control and so on. However, one usually needs to solve the coupled and nonlinear system together and simultaneously, which is very challenging. In this paper, we proposed several semismooth Newton methods to solve the nonlinear subproblems arising in image restoration in finite dimensional spaces, which leads to several highly efficient and competitive algorithms for imaging processing. With the analysis of the metric subregularities of the corresponding functions, we give both the global convergence and local linear convergence rate for the proposed augmented Lagrangian methods with semismooth Newton solvers.  
\end{abstract}


\section{Introduction}

The augmented Lagrangian method (shorted as ALM henceforth) was originated in \cite{HE,POW}. The early developments can be found in \cite{BE, FG, Roc2} and the extensive studies in infinite dimensional spaces with various applications can be found in \cite{FG, KK} and so on. The comprehensive studies of ALM for convex, nonsmooth optimization and variational problems can be found in \cite{BE, KK}. In \cite{Roc2}, the celebrated connections between the ALM and proximal point algorithms are established, where ALM is found to be equivalent to the proximal point algorithm applying to the essential dual problem. The convergence can thus be concluded in the general and powerful proximal point algorithm framework for convex optimization \cite{Roc1, Roc2}. 

ALM is very flexible for constrained optimization problems including both equality and inequality 
constraint problems \cite{BE, KK}. However, the challenging problem is solving the nonlinear and coupling systems simultaneously which usually appear while applying ALM. This is different from alternating direction method of multipliers (ADMM) type methods \cite{GL1,FG}, which decouple the unknown variables, deal with each subproblem separately and update them consecutively. 
However, for ALM, the extra effort is deserved if the nonlinear systems can be solved efficiently. This is due to the appealing asymptotic linear or superlinear convergence of ALM with increasing step sizes \cite{Roc1, Roc2, LU}. It is well known semismooth Newton methods are efficient solvers for nonsmooth and nonlinear systems. Semismooth Newton based augmented Lagrangian methods already have lots of successful applications in semidefinite programming \cite{ZST}, compressed sensing \cite{LST, ZZST}, friction and contact problem \cite{Gor, Gor1} and imaging problems \cite{HK1,KK1}. 
 
In this paper, we proposed several novel semismooth Newton based augmented Lagrangian methods for the total variation regularized image restoration problems with quadratic data terms.  Total variation (TV) regularization is  quite fundamental for imaging problems. Besides, the regularization problem is a typical nonsmooth and convex optimization problem which is very challenging to solve and is standard for testing various algorithms \cite{CP}. To the best of our knowledge, the most related work is \cite{HK1, KK1, HS}. In \cite{HK1}, Moreau-Yosida regularized dual problem of the anisotropic ROF model in Banach space is considered and solved with the semismooth Newton based ALM method. In \cite{KK1}, the primal-dual active set method is employed for nonlinear systems of the ALM applying to the additionally regularized primal problems for solving the nonlinear system.  In \cite{HS}, semismooth Newton methods are directly applied to the strongly convex regularized ROF models with strong convexity on both the primal and dual variables without ALM method.    

 All the proposed algorithms are based on applying ALM to the primal model, whose strong convexity can be employed. Compared to \cite{HS}, we do not need the strong convexity on the dual variables. Our contributions belong to the following parts. First, by introducing  auxiliary variables, we use primal-dual semismooth Newton method \cite{HS} for nonlinear systems of ALM, which is very efficient for ALM. The proposed ALM is very efficient compared to the popular existing algorithms including the primal-dual first-order method \cite{CP}. They are especially very fast for the anisotropic total variation regularized quadratic problems. Second, we proved that the maximal monotone KKT (Kuhn-Tucker) mapping is metric subregular for the anisotropic TV regularized quadratic problem.  We thus get the asymptotic linear or superlinear convergence rate of both the primal and dual sequence for the anisotropic case by the framework in \cite{Roc1, Roc2, LU}. With the calm intersection theorem \cite{KKU}, we also prove the metric subregularity of the maximal monotone operator associated with the dual problem for the isotropic case under mild condition. This leads to the asymptotic linear or superlinear convergence rate of the dual sequence for the isotropic case \cite{Roc1, Roc2, LU}. 
 To the best of our knowledge, these subregularities are novel for the TV regularized quadratic problems. Third, we also give a systematical investigation of another efficient semismooth Newton method for solving the nonlinear systems that appear in ALM, which is involved with the soft thresholding operator. 
 
 The rest of this paper is organized as follows. In section \ref{sec:rofALM}, we give an introduction to the TV regularized quadratic problems and the corresponding ALM including the isotropic and anisotropic cases. In section \ref{sec:alm:pdssn}, we present discussions on the primal-dual semismooth Newton for ALM by introducing auxiliary variables, which turns out very efficient. In section \ref{sec:alm:ssnp:thres}, we give detailed discussions on the semismooth Newton method involving the soft thresholding operators. Although all the semismooth Newton algorithms are aiming at the same nonlinear system for the isotropic or the anisotropic TV regularization. However, different formulations turn out  different algorithms with efficiency. In section \ref{sec:convergece:ssn:alm}, we give the metric subregularity for the maximal monotone KKT mapping for anisotropic TV regularized quadratic problems and for the maixmal monotone operator associated with the dual problem of isotropic TV regularized quadratic problems. Together with the convergence of the semismooth Newton method, we get the corresponding asymptotic linear convergence. In section \ref{sec:numer}, we present the detailed numerical tests for all the algorithms and the comparison with efficient first-order algorithms focused on the ROF model and deblurring model. In section \ref{sec:conclude}, we present some discussions.  
\section{Image restoration and Augmented Lagrangian Methods}\label{sec:rofALM}

The total variation regularized model is as follows.
\begin{equation}\label{eq:ROF}
\min_{u\in X} D(u) + \alpha \| \nabla u\|_{1}, \tag{P}
\end{equation}
where $D(u) = \norm[2]{Ku - z}^2/{2} + \frac{\mu}{2}\|\nabla u\|_{2}^2$ with $\mu\geq 0$ and $z$ being the known data for general image restoration \cite{HS,KK1}. Here $K$ is a linear and bounded operator depending on the problems, e.g., $Ku = \kappa \ast u$ being the convolution with kernel $\kappa$ for image deblurring problems. The strong convexity of $D(u)$ is of critical importance for the proposed algorithms and analysis.  Especially, the ROF model is for the case $K=I$ and $\mu=0$. 
For the deblurring problems, we added the strongly regularization term $\mu \| \nabla u\|_{2}^2/2$ with $\mu >0$ since $K^*K$ is usually not positive definite.
Henceforth, all the variables, operators and spaces are all in finite dimensional space setting. 
 The norm $\|\cdot\|_{1}$ based on the following  isotropic or anisotropic norm
\begin{equation}\label{eq:l1:iso:pd}
|p| :=  \sqrt{p_1^2 + p_2^2}  \ \  (\text{isotropic})\quad \text{or} \quad |p|_1 = |p_1|+|p_2| \  \ (\text{anisotropic}),
\end{equation}
where $p=(p_1, p_2)^{T} \in \mathbb{R}^2$.
We denote $|\cdot|$ as the Euclidean norm including the absolute value for real valued scalar.
Now introduce the image domain $\Omega \subset \ZZ^2$ which is 
 the discretized grid \cite{CP}
\[
\Omega = \{(i,j)\ |\ i,j \in \mathbb{N}, \ 1 \leq i \leq M,
1 \leq j \leq N\}
\]
where $M, N$ are the image dimensions. 
With discrete image domain $\Omega$, 
we define the discrete image space as
$X = \{u: \Omega \to \RR\}$ with the standard $L^2$ scalar product.
Finite differences are used to discretize the operator $\nabla$ and its
adjoint operator $\nabla^{*} = -\Div$
with homogeneous Neumann and Dirichlet boundary conditions, respectively.
We define $\nabla$ as the following operator
\[
(\nabla u)_{i,j} = ((\nabla_{1}u)_{i,j}, (\nabla_{2} u)_{i,j})^T,
\]
where forward differences are taken according to
\begin{align*}
(\nabla_{1}u)_{i,j} = \begin{cases} u_{i+1,j} -u_{i,j}, \ & \text{if} \ 1 \leq i < M, \\
0, \ & \text{if} \ i = M,
\end{cases} \quad
(\nabla_{2}u)_{i,j} = \begin{cases} u_{i,j+1} -u_{i,j}, \ & \text{if} \ 1 \leq j < N, \\
0, \ & \text{if} \ j = N.
\end{cases}
\end{align*}
With $Y = X^2$ with the standard product scalar product,
this gives a linear operator $\grad: X \to Y$.
The discrete divergence is then the negative adjoint of $\nabla$ with $\Div = -\nabla^*$,
i.e., the unique linear mapping $\Div: Y \to X$ which satisfies for $\forall u \in X$, $\ p=(p_1,p_2)^T \in Y$
\begin{equation}\label{eq:gradient:div:adjoint}
\langle \nabla u, p\rangle_{Y} = \langle u, \nabla^*p \rangle_{X} = -\langle u,   \Div p \rangle_{X}, 
\end{equation}
where the inner products are defined by
\[
 \langle u, v \rangle_X: =  \sum_{i=1}^{M}\sum_{j=1}^{N}u_{i,j}v_{i,j}, \ \ \langle p, q \rangle_Y: =  \sum_{i=1}^{M}\sum_{j=1}^{N}[(p_1)_{i,j}(q_1)_{i,j}+(p_2)_{i,j}(q_2)_{i,j}], \quad q=(q_1,q_2)^T \in Y. 
 \]
The $\Div$ operator can be computed as 
$\Div p = \nabla_{1}^{-}p_{1} + \nabla_{2}^{-}p_{2}$
and we refer to \cite{CP} for the detailed backward difference operators $\nabla_{1}^{-}$ and $ \nabla_{2}^{-}$. 

By  \eqref{eq:l1:iso:pd}, the isotropic or anisotropic TV in \eqref{eq:ROF} is as follows,
\begin{equation}\label{eq:alm:up:ani:pd1}
\|\nabla u\|_{1}: = \int_{\Omega}|\nabla u|dx, \quad \text{or} \quad \|\nabla u\|_{1}: = \int_{\Omega}|\nabla_1 u|dx + \int_{\Omega}|\nabla_2 u|dx,
\end{equation}
where $dx$ in \eqref{eq:alm:up:ani:pd1} is the area element of $\mathbb{R}^2$ and the integrals in \eqref{eq:alm:up:ani:pd1} are understood as the finite summation over all pixels. 
With these preparations, by the Fenchel-Rockafellar duality theory \cite{HBPL,KK},  the primal-dual form of \eqref{eq:ROF} can be written as
\begin{equation}
  \label{eq:tv-denoising-saddle}
  \min_{\substack{u \in X}} \max_{\substack{\lambda \in Y}} \
  D(u) + \scp[Y]{\grad u}{\lambda} -
  \mI_{\{\norm[\infty]{\lambda} \leq \alpha\}}(\lambda),
\end{equation}
where $\mI$ is the indicator function and the dual form of \eqref{eq:ROF} can be written 
 \begin{equation}\label{eq:dual:rof}
 \max_{\lambda \in Y}\left\{-\left(d(\lambda) := \frac{1}{2} \| \Div \lambda + f\|_{H^{-1}}^2 -\frac{1}{2}\|z\|_{2}^2+   \mI_{\{\norm[\infty]{\lambda} \leq \alpha\}}(\lambda)\right)\right\}. \tag{D}
 \end{equation}
The notations $f$ and the positive definite, linear operator $H$ are as follows 
 \begin{equation}\label{eq:def:H}
 f := K^*z, \quad H := -\mu \Delta + K^*K, \ \  \|u\|_{H^{-1}}^2 = \langle H^{-1}u, u\rangle_{X}, \ \  \|u\|_{H}^2 = \langle Hu, u\rangle_{X}.
 \end{equation}
 We also need the following norms for arbitrary $u \in X$, $p = (p_1,p_2) \in Y$, $1 \leq t < \infty$,
 \begin{gather*}
 \|u\|_t = \Bigl( \sum_{(i,j) \in \Omega} |u_{i,j}|^t\Bigr)^{1/t}, \quad
 \|p\|_t = \Bigl( \sum_{(i,j) \in \Omega}
 \bigl[ ((p_1)_{i,j})^2 + ((p_2)_{i,j})^2 \bigr]^{t/2}\Bigr)^{1/t}.
 \end{gather*}
 The $\norm[\infty]{\lambda}$ is defined as follows. For the isotropic case, 
 \[
 \|\lambda\|_\infty = \max_{(i,j) \in \Omega} \ \sqrt{({(\lambda_1)}_{i,j})^2 + ({\lambda_2)}_{i,j})^2}.
 \]
 For the anisotropic case, with \eqref{eq:alm:up:ani:pd1} and applying the Fenchel-Rockafellar duality theory for $\|\nabla_1 u\|_1$ and $\| \nabla_2 u\|_1$  separately, we have
 \[
 \mI_{\{\norm[\infty]{\lambda} \leq \alpha\}}(\lambda) = \mI_{\{\norm[\infty]{\lambda_1} \leq \alpha\}}(\lambda_1) +  \mI_{\{\norm[\infty]{\lambda_2} \leq \alpha\}}(\lambda_2), \quad  \|\lambda_k\|_\infty = \max_{(i,j) \in \Omega} \ |{(\lambda_k)}_{i,j}|, \ \ k=1, 2.
 \]
The optimality conditions on the saddle points $(\bar u, \bar \lambda)$ are as follows
\begin{equation}\label{eq:opti:primaldual}
 H \bar u- f + \nabla^*\bar \lambda = 0, \quad 
 \nabla \bar u \in \partial \mI_{\{\norm[\infty]{\bar \lambda} \leq \alpha\}}(\bar \lambda). 
\end{equation}
By the Fenchel-Rockafellar duality theory, we have $\bar \lambda \in \partial (\alpha \|  \nabla \bar u\|_{1})$ and
\begin{equation}
 \langle  \bar \lambda, \nabla \bar u\rangle  =  \mI_{\{\norm[\infty]{\bar \lambda} \leq \alpha\}}(\bar \lambda) + \alpha \| \nabla \bar u\|_{1}.
\end{equation}
The optimality condition for $\bar \lambda$ in \eqref{eq:opti:primaldual} is also equivalent to 
\begin{equation}
\mP_{\alpha}(\bar \lambda + c_0 \nabla \bar u) = \bar \lambda, \quad  \text{for all constant} \  c_0 >0,
\end{equation}
where $\mathcal{P}_{\alpha}$ is the projection to the feasible set $\{\lambda :\norm[\infty]{\lambda} \leq \alpha\}$, i.e.,
\begin{equation}\label{eq:projection:ani:iso}
\mathcal{P}_{\alpha}(\lambda) =\frac{p}{\max(1.0, {|\lambda|}/{\alpha})} \ \  \text{or} \ \ \mathcal{P}_{\alpha}(\lambda) =\left(\frac{\lambda_1}{\max(1, {|\lambda_1|}/{\alpha})},\frac{\lambda_2}{\max(1, {|\lambda_2|}/{\alpha})}\right)^{T}.
\end{equation}
Introducing $p:=\nabla u$, the problem \eqref{eq:ROF} becomes the following constrained optimization problem
\[
\min_{x\in X} D(u) + \alpha\|p\|_{1}, \quad \text{such that} \ \ \nabla u=p.
\]
The augmented Lagrangian method thus follows, with nondecreasing update of $\sigma_k \rightarrow c_{\infty} < +\infty$, 
\begin{align}
(u^{k+1}, p^{k+1}) &= \argmin_{u,p} L(u,p;\lambda^k) :=D(u) + \alpha \|p\|_{1} + \langle \lambda^k, \nabla u -p \rangle + \frac{\sigma_k}{2}\|\nabla u -p\|_{2}^2, \label{eq:alm:up}\\
\lambda^{k+1} &= \lambda^k + \sigma_k(\nabla u^{k+1} - p^{k+1}).\label{eq:update:lambda}
\end{align}
We will use the semismooth Newton method to solve \eqref{eq:alm:up}. The optimality conditions with fixed $\sigma_k$ and $\lambda^k$ for \eqref{eq:alm:up} are 
\begin{align}
&&Hu - f + \nabla^* \lambda^k + \sigma_k \nabla^*(\nabla u - p) = 0, \label{eq:opti:u} \\
&& \partial \alpha \| p\|_{1} - \lambda^k + \sigma_k(p-\nabla u) \ni 0,\label{eq:opti:p}
\end{align}
where $(u,p)=(u^{k+1},p^{k+1})$ are the optimal solutions of \eqref{eq:alm:up}. 
The equation \eqref{eq:opti:p} leads to
\begin{equation}\label{eq:moreau:p:update}
\lambda^k + \sigma_k \nabla u \in (\sigma_k I + \partial \alpha \|\cdot\|_{1})p \Rightarrow p  = (I + \frac{1}{\sigma_k} \partial \alpha \|\cdot\|_{1})^{-1}(\frac{\lambda^k + \sigma_k \nabla u}{\sigma_k}): = S_{\frac{\alpha}{\sigma_k}}(\frac{\lambda^k}{\sigma_k} + \nabla u),
\end{equation}
where $S_{\frac{\alpha}{\sigma_k}}(\cdot)$ is the soft thresholding operator for the isotropic or anisptropic $\|\cdot\|_1$ norm.
With relation \eqref{eq:moreau:p:update}, the augmented Lagrangian $L(u,p;\lambda^k)$ can be reformulated as
\begin{align}
&\Phi_k(u; \lambda^k, \sigma_k): = L(u, S_{\frac{\alpha}{\sigma_k}}(\frac{\lambda^k}{\sigma_k} + \nabla u); \lambda^k) \notag \\
&= D(u) + \alpha \| S_{\frac{\alpha}{\sigma_k}}(\frac{\lambda^k}{\sigma_k} + \nabla u)\|_{1}  + \frac{\sigma_k}{2}\|\frac{\lambda^k}{\sigma_k}+ \nabla u  - S_{\frac{\alpha}{\sigma_k}}(\frac{\lambda^k}{\sigma_k} + \nabla u)\|_2^2 -\frac{1}{2\sigma_k}\|\lambda^k\|_2^2, \label{eq:augmented:lagrangian:only:u}
\end{align}
which will be more convenient than \eqref{eq:alm:up} once the globalization strategy including the line search is employed.
Substituting $p$ of \eqref{eq:moreau:p:update} into \eqref{eq:opti:u}, we get
\begin{equation}\label{eq:u:alm:suntoh1}
Hu - f + \nabla^* \lambda^k + \sigma_k \nabla^*\nabla u -\sigma_k \nabla^*(I + \frac{1}{\sigma_k} \partial \alpha \|\cdot\|_{1})^{-1}(\frac{\lambda^k + \sigma_k \nabla u}{\sigma_k})=0.
\end{equation}
Denoting $G^*(p) = \alpha \|p\|_{1}$, we see the Fenchel dual function of $G^*$ is $G(h) = I_{\{ \|\cdot\|_{\infty} \leq \alpha\}}(h)$. With the Moreau's identity, 
\begin{equation}\label{eq:moreau:indentity}
x = (I + \tau \partial G)^{-1}(x) + \tau (I  +\frac{1}{\tau} \partial G^*)^{-1}(\frac{x}{\tau}),
\end{equation}
we arrive at
\begin{equation}\label{eq:moreau:sub}
\sigma_k p = \sigma_k(I + \frac{1}{\sigma_k} \partial \alpha \|\cdot\|_{1})^{-1}(\frac{\lambda^k + \sigma_k \nabla u}{\sigma_k})=\lambda^k + \sigma_k \nabla u  - (I + \sigma_k \partial G)^{-1}(\lambda^k + \sigma_k \nabla u ).
\end{equation}
Substituting $p$ of \eqref{eq:moreau:sub} into \eqref{eq:opti:u} leads to the equation of $u$
\begin{equation}\label{eq:u:proj:solve}
Hu - f +   \nabla ^*   (I + \sigma_k \partial G)^{-1}(\lambda^k + \sigma_k \nabla u ) = 0.
\end{equation}
Indeed, we can solve \eqref{eq:u:alm:suntoh1} directly with semismooth Newton methods, which will be discussed in the subsequent sections. Now let's turn to another formulation by introducing an auxiliary variable
\begin{equation}\label{eq:opti:modi:pro}
h: = (I + \sigma_k \partial G)^{-1}(\lambda^k + \sigma_k \nabla u ) = \mathcal{P}_{\alpha}(\lambda^k + \sigma_k \nabla u).
\end{equation}
By the definition of the projection \eqref{eq:projection:ani:iso} and taking the isotropic  case for example,  \eqref{eq:opti:modi:pro} becomes
\begin{equation}\label{eq:proj:cons}
h = \mathcal{P}_{\alpha}(\lambda^k + \sigma_k \nabla u) = \dfrac{\lambda^k + \sigma_k \nabla u}{\max(1.0, {|\lambda^k + \sigma_k \nabla u|}/{\alpha})}.
\end{equation}
The equation \eqref{eq:u:proj:solve} thus becomes
\begin{equation}\label{eq:opti:modi:h}
Hu - f + \nabla^*h = 0.
\end{equation}
Combining \eqref{eq:opti:modi:pro}, \eqref{eq:proj:cons} and \eqref{eq:opti:modi:h}, we get the following equations of $(u, \lambda)$ instead of $(u,p)$, 
\begin{equation}\label{eq:opti:u:lambda}
\mathcal{F}(u,h) = \begin{bmatrix} 0 \\ 0 \end{bmatrix}, \quad \mathcal{F}(u,h):=
\begin{bmatrix}
Hu - f + \nabla^*h \\
- \sigma_k \nabla u -  \lambda^k + \max \big(1.0, \dfrac{|\lambda^k + \sigma_k \nabla u|}{\alpha} \big) h
\end{bmatrix}
.
\end{equation}
Once solving any of the equations \eqref{eq:u:alm:suntoh1} and \eqref{eq:opti:u:lambda} and obataining the solution $u$, we can update the $p$ by \eqref{eq:moreau:p:update} or \eqref{eq:moreau:sub} accordingly. Then the Lagrangian multiplier $\lambda^{k+1}$ can be updated by \eqref{eq:update:lambda}. Actually, according to \eqref{eq:moreau:sub},
compared to the update of $\lambda^{k+1}$ \eqref{eq:update:lambda}, we see the update of the multiplier can also be
\begin{equation}\label{eq:update:multiplier:projection}
\lambda^{k+1} = (I + \sigma_k \partial G)^{-1}(\lambda^k + \sigma_k \nabla u )=\mathcal{P}_{\alpha}(\lambda^k + \sigma_k \nabla u),
\end{equation}
which is a nonlinear update compared to the linear update \eqref{eq:update:lambda}.  We refer to \cite{KK} (chapter 4) for general nonlinear updates of Lagrangian multipliers with different derivations and framework of ALM. 
Throughout this paper,  we will use the semismooth Newton methods to any of the subproblems \eqref{eq:u:alm:suntoh1} and \eqref{eq:opti:u:lambda}.  The different formulations of \eqref{eq:u:alm:suntoh1} and \eqref{eq:opti:u:lambda} will bring out different algorithms and it would turn out different efficiency.
 We begin with the semismoothness where the Newton derivative can be chosen as the Clarke generalized derivative \cite{KK,KKU1}. 
\begin{definition}[Newton differentiable and Newton Derivative \cite{KK}] $F: D \subset X \rightarrow Z$ is called Newton differentiable at $x$ if there exist an open neighborhood $N(x) \subset D$ and mapping $G: N(x)\rightarrow \mathcal{L}(X,Z)$ such that (Here the spaces $X$ and $Z$ are Banach spaces.)
	\begin{equation}
	\lim_{|h|\rightarrow 0} \frac{|F(x+h)-F(x)-G(x+h)h|_{Z}}{|h|_X}=0.
	\end{equation}
	The family $\{G(s): s \in N(x)\}$ is called an Newton derivative of $F$ at $x$.
\end{definition}
If $F:\mathbb{R}^n \rightarrow \mathbb{R}^m$ and the set of mapping $G$ is the Clarke generalized derivative $\partial F$, we call $F$ is semismooth \cite{KKU}. 
\begin{definition}[Semismoothness \cite{RM, LST, MU}] Let $F: O \subseteq X \rightarrow Y$ be a locally Lipschitz continuous function on the open set $O$. $F$ is said to be semismooth at $x \in O$ if $F$ is directionally differentiable at $x$ and for any $V\in \partial F(x+ \Delta x)$ with $\Delta x \rightarrow 0$,
	\[
	F(x+\Delta x) -F(x) - V\Delta x = {o}(\|\Delta x\|).
	\]  
\end{definition}
 The Newton derivatives of vector-valued functions can be computed component-wisely \cite{MU} (Proposition 2.10) or \cite{Cla} (Theorem 9.4). Together with the definition of semismoothness, we have the following lemma.
\begin{lemma}\label{lem:vector:semismooth:newton}
	Suppose $F : \mathbb{R}^n \rightarrow \mathbb{R}^m$ and $F = (F_1(x), F_2(x), \cdots, F_l(x))^{T}$ with $F_i : \mathbb{R}^n \rightarrow \mathbb{R}^{l_i} $  being semismooth. Here $l_i \in \mathbb{Z}^{+}$ and $\sum_{i=1}^l l_i=m$.  Assuming the Newton derivative $ D_N F_i(x) \in \partial F_i(x)$, $i  = 1,2,\cdots, l$, then the Newton derivative of $F$ can be chosen as
	\begin{equation}
	D_N F(x) = \begin{bmatrix}
	D_N F_1(x), D_N F_2(x)
, \cdots, D_N F_l(x)	\end{bmatrix}^T.
	\end{equation}
\end{lemma}
Now we turn to the semismooth Newton method for solving \eqref{eq:u:alm:suntoh1} and \eqref{eq:opti:u:lambda}. The semismooth Newton method for the general nonlinear equation $F(x)=0$  can be written as 
\begin{equation}\label{semi:smoothnewton:cal:newton:direc}
\mV (x^l)^{-1}\delta x^{l+1} =  - F (x^l),
\end{equation}
where $\mV(x^l)$ is a semismooth Newton derivative of $F$ at $x^l$, for example $\mV(x^l) \in \partial F(x^l)$.  Additionally, $\mV(x)$ needs to satisfy the \emph{regularity condition} \cite{Cla, MU}. Henceforth, we say   $\mV(x)$ satisfies the \emph{regularity condition} if  $\mV(x)^{-1}$ exist and are uniformly bounded in a small neighborhood of the solution $x^*$ of $F(x^*)=0$. When the globalization strategy including line search is necessary, 
one can get the Newton update $x^{l+1}$ with the Newton direction $\delta x^{l+1}$ in \eqref{semi:smoothnewton:cal:newton:direc}.
Once the globalization strategy is not needed, the semismooth Newton iteration can also be written as
follows with solving $x^{l+1}$ directly
\begin{equation}\label{semi:smoothnewton:sys}
\mV(x^l)x^{l+1} = \mV(x^l)x^l  - F(x^l).
\end{equation}
\section{Augmented Lagrangian Method with Primal-Dual Semismooth Newton Method} \label{sec:alm:pdssn}
\subsection{The Isotropic Total Variation}

Now we turn to the semismoothness of the nonlinear system \eqref{eq:opti:u:lambda}. The only nonlinear or nonsmooth part comes from the function $\max (1.0, |\lambda^k + \sigma_k \nabla u|/\alpha)$.
\begin{lemma}\label{lem:semismooth:max}
The function $\mG(u):=\max (1.0, \dfrac{|\lambda^k + \sigma_k \nabla u|}{\alpha})$  is semismooth on $u$ and its Clarke's generalized gradient for $u$ is as follows,
\begin{equation}\label{eq:newton:deri:max:iso}
\left\{\chi^s_{\lambda^k, u}\dfrac{\sigma_k}{\alpha}\dfrac{\langle \lambda^k + \sigma_k \nabla u,  \nabla \cdot \ \rangle }{ |\lambda^k + \sigma_k \nabla u |} \ | \ s\in[0,1]\right\} = \partial_{u}(\max (1.0, \dfrac{|\lambda^k + \sigma_k \nabla u|}{\alpha})),
\end{equation}
where $\chi^s_{\lambda^k, u}$ is the generalized Clarke derivatives of $\max(\cdot, 1.0)$,
\begin{equation}\label{eq:newton:deri:max}
\chi^s_{\lambda^k, u} = \begin{cases}
1, \quad & |\lambda^k +\sigma_k \nabla u | /\alpha >1.0, \\
s, \quad & |\lambda^k +\sigma_k \nabla u | /\alpha=1.0, \  s\in [0,1], \\ 
0, \quad & |\lambda^k  + \sigma_k \nabla u | / \alpha <1.0.
\end{cases}
\end{equation}
Furthermore,  $\mF(u,h)$  in \eqref{eq:opti:u:lambda} is semismooth on $(u,h)$. Henceforth, we choose $s=1$ for the Newton derivative of $\mG(u)$ in \eqref{eq:newton:deri:max:iso}  and denote $\chi_{\lambda^k, u} := \chi^1_{\lambda^k, u}$.
\end{lemma}
\begin{proof}
	 We claim that $\mG(u)$ is actually a $PC^{\infty}$ function of $u$. Introduce $\mG_1(u) = 1.0$ and $\mG_2(u) = {|\lambda^k + \sigma_k \nabla u|}/{\alpha}$ which are \emph{selection functions} of $\mG(u)$. $\mG(u)$ is called as the \emph{continuous selection} of the functions $\mG_1(u)$ and $\mG_2(u)$ \cite{Sch} (Chapter 4) (or Definition 4.5.1 of \cite{FP}). We see $\mG_1(u)$ is a smooth function  and $\mG_2(u)$ is smooth in any open set outside the closed set $D_0:=\{u \ | \ |\lambda^k + \sigma_k \nabla u|=0\}$. Thus for any $u\in D_{\alpha}:=\{u \ | \ |\lambda^k + \sigma_k \nabla u|=\alpha\}$, there exists a small open neighborhood of $u$ such that $\mG_1(u)$ and $\mG_2(u)$ are smooth functions.  We thus conclude that $\mG(u)$ is a $PC^{\infty}$ function of $u$. $\mG(u)$ is also $PC^1$ and hence is semismooth on $u$ \cite{MU} (Proposition 2.26), since $\nabla_u \mG_1(u) = 0$ and $\nabla_u \mG_2(u) = {\sigma_k\langle \lambda^k + \sigma_k \nabla u,  \nabla \cdot \ \rangle }/{ (\alpha|\lambda^k + \sigma_k \nabla u |)}$ out side $D_0$. For any $u \in D_{\alpha}$, by \cite{Sch} (Proposition 4.3.1), we see
	\[
	\partial_u \mG(u) = \text{co}\{\nabla_u \mG_1(u), \nabla_u \mG_2(u)\},
	\]
where ``$\text{co}$" denotes the convex hull of the corresponding set \cite{CL}.
   We thus obtain the equation \eqref{eq:newton:deri:max:iso}. Since each component of $\mathcal{F}(u,h)$ is an affine function on $h$ and is semismoooth on $(u,h)$, the semismooth of $\mathcal{F}(u,h)$ on $(u,h)$ then follows \cite{MU} (Proposition 2.10). 
\end{proof}
 By Lemma \ref{lem:vector:semismooth:newton} and \ref{lem:semismooth:max}, denoting $x^l = (u^l, h^l)$, $x = (u, h)$ and $\mathcal{F}(u,h) = (\mF_1(u,h), \mF_2(u,h))^{T}$, the Newton derivative of $F(u,h)$ can be chosen as
 \begin{equation}\label{eq:newton:vector:deri}
 D_{N}\mF = (D_N \mF_1(u,h), D_N \mF_2(u,h))^{T}.
 \end{equation}
Thus the Newton derivative of the nonlinear equation \eqref{eq:opti:u:lambda} can be chosen as 
\begin{equation}
\mV^{I}(x^l) = \begin{bmatrix}
H & \nabla^* \\
-\sigma_k \nabla + \chi_{\lambda^k, u^l}\dfrac{\sigma_k}{\alpha}\dfrac{\langle \lambda^k + \sigma_k \nabla u^l,  \nabla \cdot \ \rangle    }{ |\lambda^k + \sigma_k \nabla u^l |}h^l & \max \big(1.0, \dfrac{|\lambda^k + \sigma_k \nabla u^l|}{\alpha}\big)
\end{bmatrix}.
\end{equation}
Let's introduce
\begin{align*}
&D_l = 
U_{\sigma_k}(\lambda^k,u^l): = \max \big(1.0, \dfrac{|\lambda^k + \sigma_k \nabla u^l|}{\alpha}\big),
\quad 
B_l = \begin{bmatrix}
\chi_{\lambda^k, u^l}\dfrac{\sigma_k}{\alpha}\dfrac{\langle \lambda^k + \sigma_k \nabla u^l,  \nabla \cdot \rangle }{ |\lambda^k + \sigma_k \nabla u^l |}h^l
\end{bmatrix},\\
&C_l = -\sigma_k\nabla +B_l.
\end{align*}
It can be readily verified that 
\[
\mV^{I}(x^l)x^l - \mathcal{F}(x^l) =\begin{bmatrix}
f \\ \lambda^k + \chi_{\lambda^k, u^l}\dfrac{\sigma_k}{\alpha}\dfrac{\langle \lambda^k + \sigma_k \nabla u^l,  \nabla u^l \rangle }{ |\lambda^k + \sigma_k \nabla u^l |}h^l 
\end{bmatrix}=\begin{bmatrix} f \\ \lambda^k + B_lu^l \end{bmatrix}:= \begin{bmatrix} f \\ b_2^l \end{bmatrix}.
\]
Next, we turn to solve the Newton update \eqref{semi:smoothnewton:sys}
\begin{equation}\label{eq:system:u:h}
\begin{bmatrix}
H & \nabla^* \\
C_l& D_l
\end{bmatrix}
\begin{bmatrix}
u^{l+1} \\ h^{l+1}
\end{bmatrix}
=\begin{bmatrix} f \\ b_2^l \end{bmatrix}.
\end{equation}
For solving the linear system \eqref{eq:system:u:h}, it is convenient to solve $u^{l+1}$ first, i.e., solving the equation of the Schur complement $\mV^I(x^l)/D_l$.   Substituting 
\begin{equation}\label{eq:ufirst:h}
h^{l+1} =  \bigg( b_2^l  + \sigma_k \nabla u^{l+1} - \chi_{\lambda^k, u^l}\dfrac{\sigma_k}{\alpha}\dfrac{\langle \lambda^k + \sigma_k \nabla u^l,  \nabla u^{l+1} \rangle } { |\lambda^k + \sigma_k \nabla u^l |}h^l \bigg) \bigg/ \bigg(\max \big(1.0, \dfrac{|\lambda^k + \sigma_k \nabla u^l|}{\alpha}\big) \bigg)
\end{equation}
into the first equation on $u^{k+1}$,  we have
\begin{equation}\label{eq:newton:equation:iso}
(H-\nabla^*D_l^{-1}C_l)u^{l+1} =  f + \Div \dfrac{b_2^l}{U_{\sigma_k}( \lambda^k, u^l)},
\end{equation}
which is also the following equation in detail
\begin{align}\label{eq:calculate:u:first}
H u^{l+1} - \Div \dfrac{\sigma_k\nabla u^{l+1} - \chi_{\lambda^k, u^l}\dfrac{\sigma_k}{\alpha}\dfrac{\langle \lambda^k + \sigma_k \nabla u^l,  \nabla u^{l+1} \rangle } { |\lambda^k + \sigma_k \nabla u^l |}h^l}{U_{\sigma_k}( \lambda^k, u^l)} = f + \Div \dfrac{b_2^l}{U_{\sigma_k}( \lambda^k, u^l)}.
\end{align}
After calculating $u^{l+1}$ in \eqref{eq:calculate:u:first}, we get $h^{l+1}$ by \eqref{eq:ufirst:h}.

We can also first calculate $h^{l+1}$ following the calculation of $u^{l+1}$, i.e., solving the equation of the Schur complement $\mV^I(x^l)/H$.
Solving dual variables first can also be found in \cite{HPRS}, where the primal-dual semismooth Newton is employed for total generalized variation.
 Substituting
\begin{equation}\label{eq:u:h:recover}
 u^{l+1} = H^{-1}(\Div h^{l+1} + f),
\end{equation} 
  into \eqref{eq:ufirst:h}, we obtain the linear equation of $h^{l+1}$
  \begin{equation}
  (D_l-C_l H^{-1} \nabla^*)h^{l+1} = b_2^k - C_l H^{-1} f,
  \end{equation}
  which is the folloing equation in detail,
\begin{align}
&\max \big(1.0, \dfrac{|\lambda^k + \sigma_k \nabla u^l|}{\alpha}\big)h^{l+1} -\sigma_k \nabla H^{-1} \Div h^{l+1}+\chi_{\lambda^k, u^l}\dfrac{\sigma_k}{\alpha}\dfrac{\langle \lambda^k + \sigma_k \nabla u^l,  \nabla H^{-1}  \Div h^{l+1} \rangle } { |\lambda^k + \sigma_k \nabla u^l |}h^l \notag \\
&= b_2^l -C_{l} H^{-1}f \label{eq:ssn:pdd:h}
\end{align}
We then recover $u^{l+1}$ by \eqref{eq:u:h:recover} after calculating $h^{l+1}$.


We have the following lemma for the regularity conditions of the Newton derivative.
\begin{lemma}\label{lem:positive:iso:pd}
	If the feasibility of $h^l$ is satisfied, i.e., $|h^l|\leq \alpha$ by \eqref{eq:opti:modi:pro}, we have the positive definiteness of the Schur complement $\mV^I(x^l)/D_l=(H-\nabla^*D_l^{-1}C_l)$,
	\begin{equation}\label{eq:positive:pd:iso}
	\langle -\nabla^*D_l^{-1}C_lu, u  \rangle_{X} \geq 0 \Rightarrow \langle (\mV^I(x^l)/D_l)u,u \rangle_{X} \geq \|u\|_{H}^2 = \langle Hu,u \rangle_X .
	\end{equation}
	Thus $\mV^I(x^l)/D_l$ satisfies the regularity condition. Furthermore $\mV^I(x^l)$ can be chosen as a Newton derivative of $\mF(u,h)$. The linear operator $\mV^I(x^l)$ and the Schur complement $\mV^I(x^l)/H$ satisfy the regularity condition for any fixed $\sigma_k$ and $\lambda^k$, i.e.,  $\mV^I(x^l)$ and $\mV^I(x^l)/H$ are nonsingular and the corresponding inverse are uniformly bounded for any fixed $\sigma_k$ and $\lambda^k$. 
\end{lemma}
\begin{proof}
	We first prove the regular condition of $\mV^I(x^l)/D_l$, whose proof is essentially similar to the proof of Lemma 3.3 in \cite{HS}. 
	Denote 
	\[
	\Omega^{+}: = \{ (x_1,x_2) \in \Omega: |\lambda^k + \sigma_k\nabla u | \geq \alpha\}, \quad 	\Omega^{-}: = \{ (x_1,x_2) \in \Omega: |\lambda^k + \sigma_k\nabla u | < \alpha\}, \quad \Omega = \Omega^{+} \cup \Omega^{-}.
	\]
	Since for any $u\in X$, we have
	\begin{align}
	&\langle -\nabla^*D_l^{-1}C_lu, u  \rangle_{X} =-\langle D_l^{-1}C_lu, \nabla u \rangle_{Y}
	=\langle -D_l^{-1}(-\sigma_k \nabla + B_l)u, \nabla u\rangle_{Y} \notag \\
	&=\langle \sigma_k D_l^{-1}\nabla u, \nabla u \rangle_{2} - \langle D_l^{-1}B_lu, \nabla u \rangle_{Y} \label{eq:inequ:iso:posi:esi}\\
	&=\int_{\Omega^{-}}\sigma_kD_l^{-1}|\nabla u|^2 dx +  \int_{\Omega^{+}}\sigma_kD_l^{-1}|\nabla u|^2 dx-\int_{\Omega^{+}} \langle D_l^{-1} \dfrac{\sigma_k}{\alpha}\dfrac{\langle \lambda^k + \sigma_k \nabla u^l,  \nabla u \rangle }{ |\lambda^k + \sigma_k \nabla u^l |}h^k,\nabla u \rangle dx. \notag
	\end{align}
	With the assumption $|h^l| \leq \alpha$ and direct calculations, we see
	\begin{equation}\label{eq:inequ:iso:posi:esiII}
	|\langle D_l^{-1} \dfrac{\sigma_k}{\alpha}\dfrac{\langle \lambda^k + \sigma_k \nabla u^l,  \nabla u \rangle }{ |\lambda^k + \sigma_k \nabla u^l |}h^l,\nabla u \rangle |
	\leq D_l^{-1} \dfrac{\sigma_k}{\alpha}\dfrac{|\lambda^k + \sigma_k \nabla u^l|| \nabla u|  }{ |\lambda^k + \sigma_k \nabla u^l|}|h^l||\nabla u|
	\leq  \sigma_k D_l^{-1} |\nabla u|^2.
	\end{equation}
	Combining \eqref{eq:inequ:iso:posi:esi} and \eqref{eq:inequ:iso:posi:esiII}, we arrive at
	\[
	\langle -\nabla^*D_l^{-1}C_lu, u  \rangle_{X} \geq \int_{\Omega^{-}}\sigma_kD_l^{-1}|\nabla u|^2 dx \geq 0,
	\]
	which leads to \eqref{eq:positive:pd:iso}.  We thus conclude that  $\mV^I(x^l)/D_l \geq H$ and $(\mV^I(x^l)/D_l)^{-1} $ is bounded by the positive definiteness of $H$. For the regularity condition of ${\mV^I}(x^l)$, it is known that (\cite{Zhangfu} formula 0.8.1 which is similar to the Banachiewicz inversion formula)
	\[
	{\mV^I}(x^l)^{-1} = \begin{bmatrix}
	(\mV^I(x^l)/D_l)^{-1} & -(\mV^I(x^l)/D_l)^{-1}\nabla^*D_l^{-1} \\
	-D_l^{-1}C_l 	(\mV^I(x^l)/D_l)^{-1} & D_l^{-1} + D_l^{-1}C_l	(\mV^I(x^l)/D_l)^{-1}\nabla^*D_l^{-1}.
	\end{bmatrix}
	\]
	By the boundedness of $(\mV^I(x^l)/D_l)^{-1}$, $C_l$ and $D_l^{-1}$, we get the boundedness of ${\mV^I}(x^l)^{-1}$. The boundedness of  $D_l^{-1}$ is because $D_l \geq I$ and the boundedness of $C_l$ comes from the boudnedness of $\nabla$ and $B_l$.
	
	Similarly, for the existence and boundedness of $(\mV^I(x^l)/I)^{-1}$,   by the Duncan inversion formula (see 0.8.1 and 0.8.2 of \cite{Zhangfu}) or Woodbury formula, we have
	\[
	(\mV^I(x^l)/H)^{-1} = (D_l -C_lH^{-1}\nabla^*)^{-1} = D_l^{-1} + D_l^{-1}C_l	(\mV^I(x^l)/D_l)^{-1}\nabla^*D_l^{-1}.
	\]
	We thus get the boundedness of $(\mV^I(x^l)/H)^{-1}$.
\end{proof}

\subsection{The Anisotropic Total Variation}
For the anisotropic $l_1$ norm in \eqref{eq:l1:iso:pd}, with \eqref{eq:projection:ani:iso}, the projection \eqref{eq:proj:cons} becomes 
\begin{equation}\label{eq:proj:cons:ani}
\mathcal{P}_{\alpha}^{A}(\lambda^k + \sigma_k \nabla u) = \left(\dfrac{\lambda_1^k + \sigma_k \nabla_1 u}{\max(1.0, {|\lambda_1^k + \sigma_k \nabla_1 u|}/{\alpha})}, \dfrac{\lambda_2^k + \sigma_k \nabla_2 u}{\max(1.0, {|\lambda_2^k + \sigma_k \nabla_2 u|}/{\alpha})}\right)^T.
\end{equation}
With similar analysis as the isotropic case, the equation \eqref{eq:opti:u:lambda} becomes
\begin{equation}\label{eq:opti:u:lambda:ani}
\mF^A(u,h) = \begin{bmatrix} 0 \\ 0 \\0 \end{bmatrix}, \quad \mF^A(u,h) := 
\begin{bmatrix}
Hu - f + \nabla^*h \\
- \sigma_k \nabla_1 u -  \lambda_1^k + \max \big(1.0, \dfrac{|\lambda_1^k + \sigma_k \nabla_1 u|}{\alpha} \big) h_1 \\
- \sigma_k \nabla_2 u -  \lambda_2^k + \max \big(1.0, \dfrac{|\lambda_2^k + \sigma_k \nabla_2 u|}{\alpha} \big) h_2
\end{bmatrix}
.
\end{equation}
Similar to Lemma \ref{lem:semismooth:max}, we have the following lemma, whose proof is completely similar to Lemma \ref{lem:semismooth:max} and we omit here.
\begin{lemma}\label{lem:ani:pdd}
	The functions $\max (1.0, {|\lambda_1^k + \sigma_k \nabla_1 u|/\alpha})$ and $\max (1.0, {|\lambda_2^k + \sigma_k \nabla_2 u|}/{\alpha})$ are semismooth functions of $u$ and their Clarke generalized gradients are as follows,
	\begin{equation}\label{eq:subdifffer:ani:pd:max}
	\left\{\chi_{\lambda^k, u}^{i,s}\dfrac{\sigma_k}{\alpha}\dfrac{\langle \lambda_i^k + \sigma_k \nabla_i u,  \nabla_i \cdot \ \rangle }{ |\lambda_i^k + \sigma_k \nabla_i u |} \ | \ s\in [0,1] \right \} = \partial_{u}(\max (1.0, \dfrac{|\lambda_i^k + \sigma_k \nabla_i u|}{\alpha})), \quad i=1,2
		\end{equation}
	where $\chi_{\lambda^k, u}^{1,s}$ and $\chi_{\lambda^k, u}^{2,s}$ are the generalized derivatives of $\max(\cdot, 1.0)$,
	\begin{equation}\label{eq:ani:pd:actset}
	\chi_{\lambda^k, u}^{1,s} = \begin{cases}
	1, \quad & |\lambda_1^k +\sigma_k \nabla_1 u | / \alpha > 1.0, \\
s, \quad & |\lambda_1^k +\sigma_k \nabla_1 u | / \alpha = 1.0, \\
	0, \quad & |\lambda_1^k  + \sigma_k \nabla_1 u | / \alpha < 1.0,
	\end{cases} \quad 
		\chi_{\lambda^k, u}^{2,s} = \begin{cases}
	1, \quad & |\lambda_2^k +\sigma_k \nabla_2 u | / \alpha \geq 1.0, \\
   s, \quad & |\lambda_2^k +\sigma_k \nabla_2 u | / \alpha = 1.0, \ s \in [0,1],\\
	0, \quad & |\lambda_2^k  + \sigma_k \nabla_2 u | / \alpha < 1.0.
	\end{cases}
	\end{equation}
	Furthermore $F^A(u,h)$ is semismooth on $(u,h)$. Henceforth, denote $\chi_{\lambda^k, u}^{1}:=\chi_{\lambda^k, u}^{1,1}$ and $\chi_{\lambda^k, u}^{2}: = \chi_{\lambda^k, u}^{2,1} $ for $s=1$ cases in \eqref{eq:ani:pd:actset}.
\end{lemma}
Let's introduce
\[
D_l^{A} = \begin{bmatrix}
U_{\sigma_k}^1 & 0\\
0 & U_{\sigma_k}^2
\end{bmatrix},
\quad 
B_l^{A} = \begin{bmatrix}
\chi_{\lambda^k, u^k}^1\dfrac{\sigma_k}{\alpha}\dfrac{\langle \lambda_1^k + \sigma_k \nabla_1 u^l,  \nabla_1  \cdot \  \rangle } { |\lambda_1^k + \sigma_k \nabla_1  u^l |}h_1^l  \\
\chi_{\lambda^k, u^l}^2\dfrac{\sigma_k}{\alpha}\dfrac{\langle \lambda_2^k + \sigma_k \nabla_2 u^l,  \nabla_2  \cdot \ \rangle } { |\lambda_2^k + \sigma_k \nabla_1 u^l |}h_2^l
\end{bmatrix}, \ \ C_l^{A} = -\sigma \nabla + B_l^A,
\]
where 
\[
U_{\sigma_k}^1( \lambda^k, u^l)=\max \left(1.0, \dfrac{|\lambda_1^k + \sigma_k \nabla_1 u^l|}{\alpha}\right),\quad 
U_{\sigma_k}^2( \lambda^k, u^l)=\max \left(1.0, \dfrac{|\lambda_2^k + \sigma_k \nabla_2 u^l|}{\alpha}\right).
\]
By Lemma \ref{lem:vector:semismooth:newton} and \ref{lem:ani:pdd},  since $F^A(u,h)$ is an affine function of $h$, together with Lemma \ref{lem:ani:pdd}, it can be readily verified that, we can choose the Newton derivative of the nonlinear equation \eqref{eq:opti:u:lambda:ani} as
\begin{equation}
\mV^{A} = \begin{bmatrix}
H & \nabla ^* \\
C_l^{A} & D_l^{A}
\end{bmatrix}.
\end{equation}
The right-hand side becomes
\[
\mV^A(x^l)x^l - \mF^A(x^l) =\begin{bmatrix}
f \\ B_lf 
\end{bmatrix}
=
\begin{bmatrix}
f \\ \lambda_1^k + \chi_{\lambda^k, u^l}^1\dfrac{\sigma_k}{\alpha}\dfrac{\langle \lambda_1^k + \sigma_k \nabla_1 u^l,  \nabla_1 u^l \rangle }{ \|\lambda_1^k + \sigma_k \nabla_1 u^l \|}h_1^l 
 \\ \lambda_2^k + \chi_{\lambda^k, u^l}^2\dfrac{\sigma_k}{\alpha}\dfrac{\langle \lambda_2^k + \sigma_k \nabla_2 u^l,  \nabla_2 u^l \rangle }{ \|\lambda_2^k + \sigma_k \nabla_2 u^l \|}h_1^l 
\end{bmatrix}: = \begin{bmatrix} f \\ b_1^l \\b_2^l \end{bmatrix}.
\]
For solving $u^{k+1}$ first, the Newton update becomes
\begin{equation}\label{eq:newton:equation:ani}
(H-\nabla^*{{D^A_l}}^{-1}{C^{A}_l})u^{l+1} =  f + \Div {D^{A}_l}^{-1}b^l,
\end{equation}
where $b^l = (b_1^l, b_2^l)^{T}$. Then $h^{l+1}$ can be recovered by 
\begin{equation}\label{eq:ufirst:h:ani}
h^{l+1} =  \begin{bmatrix}\left( b_1^l  + \sigma_k \nabla_1 u^{l+1} - \chi_{\lambda^k, u^l}^1\dfrac{\sigma_k}{\alpha}\dfrac{\langle \lambda_1^k + \sigma_k \nabla_1 u^l,  \nabla_1 u^{l+1} \rangle } { |\lambda_1^k + \sigma_k \nabla_1 u^l |}h_1^l \right) \bigg / U_{\sigma_k}^1( \lambda^k, u^l) \\
\left( b_2^l  + \sigma_k \nabla_2 u^{l+1} - \chi_{\lambda^k, u^l}^2\dfrac{\sigma_k}{\alpha}\dfrac{\langle \lambda_2^k + \sigma_k \nabla_2 u^l,  \nabla_2 u^{l+1} \rangle } { |\lambda_2^k + \sigma_k \nabla_2 u^l |}h_2^l \right) \bigg/ U_{\sigma_k}^2( \lambda^k, u^l)
\end{bmatrix}.
\end{equation}
For calculating $h^{+1}$ first with solving the Schur complement $\mV^A(x^l)/H$,   we have 
\begin{equation}\label{eq:pdssn:h:ani}
(D_l^A - C_l  H^{-1} \nabla^*)h^{l+1}=  b^l -C_{l}H^{-1}f.
\end{equation}
Then $u^{l+1}$ can be recovered through
\[
u^{l+1}=H^{-1} (f-\nabla^*h^{l+1}).
\]
For the positive definiteness and regularity condition of the Schur complement $\mV^A(x^l)/D_l^A$ , we have the following lemma, whose proof is completely similar to the proof of Lemma \ref{lem:positive:iso:pd} and we omit here.
\begin{lemma}\label{lem:positive:ani:pd} 
	If the feasibility of $h^l$ is satisfied, i.e., $|h_i^l|\leq \alpha$ by \eqref{eq:opti:modi:pro}, $i=1,2$ we have the positive definiteness of $(H-\nabla^*{D^A_l}^{-1}C^A_l)$,i.e.,
	\begin{equation}\label{eq:positive:pd:ani:iso}
	\langle -\nabla^*{D^A_l}^{-1}C^A_l u, u  \rangle_{X} \geq 0 \Rightarrow \langle (H-\nabla^*{D^A_l}^{-1}C^A_l)u,u \rangle_{X} \geq \|u\|_{H}^2.
	\end{equation}
	We thus conclude $\mV^A(x^l)$ can be chosen as a Newton derivative of $\mF^A(u,h)$.
	$\mV^A(x^l)/D_l^A$ satisfies the regularity condition. The linear operator $\mV^A(x^l)$ and the Schur complement $\mV^A(x^l)/H$ satisfy the regularity condition for any fixed $\sigma_k$ and $\lambda^k$. 	
\end{lemma}
We conclude the section \ref{sec:alm:pdssn} by the following algorithm \ref{alm:SSN_PDP} and \ref{alm:SSN_PDD} which are subproblems for the $k$th iteration of ALM applying to \eqref{eq:ROF}. 
\begin{algorithm}[h]
	\caption{Primal-dual semismooth Newton method with solving the primal variable first (SSNPDP) for  \eqref{eq:u:proj:solve} with auxiliary \eqref{eq:opti:u:lambda} or \eqref{eq:opti:u:lambda:ani}
		\label{alm:SSN_PDP}}
	\begin{algorithmic}
		\STATE {Given corrupted image $f$,  multiplier $\lambda^k$, step size $\sigma_k$ of ALM, $u^0$, auxiliary  variable $h^0$ in the feasible set  $\{h:\norm[\infty]{h} \leq \alpha\}$. \\
		Iterate the following steps for $l=0, 1, \cdots, $ unless some stopping criterion associated with the nonlinear system \eqref{eq:opti:u:lambda} is satisfied.}
		\STATE {\textbf{Step 1: }Solve the linear system \eqref{eq:calculate:u:first} for $u^{l+1}$ with some stopping criterion for the isotropic case (or \eqref{eq:newton:equation:ani} for the anisotropic case) with iterative method (BiCGSTAB):}
		\STATE{\textbf{Step 2: }Update $h^{l+1}$ by \eqref{eq:ufirst:h}} (or \eqref{eq:ufirst:h:ani} for the anisotropic case)
		\STATE{\textbf{Step 3: }Project $h^{l+1}$ to the feasible set  $\{h:\norm[\infty]{h} \leq \alpha\}$, i.e., $h^{l+1}  =\mathcal{P}_{\alpha}(h^{l+1})$. Set $(u^{l+1},h^{l+1})$ as the initial value for the next Newton iteration and go to \textbf{Step 1}}.
		\STATE{Output $u^{k+1}$, $h^{k+1}$}
	\end{algorithmic}
\end{algorithm} 

\begin{algorithm}[h]
	\caption{Primal-dual semismooth Newton method with solving the dual variable first (SSNPDD) for  \eqref{eq:u:proj:solve} with auxiliary \eqref{eq:opti:u:lambda} or \eqref{eq:opti:u:lambda:ani}
		\label{alm:SSN_PDD}}
	\begin{algorithmic}
		\STATE {Given corrupted image $f$,  multiplier $\lambda^k$, step size $\sigma_k$ of ALM, $u^0$, auxiliar  variable $h^0$ in the feasible set  $\{h:\norm[\infty]{h} \leq \alpha\}$. \\
			Iterate the following steps for $l=0, 1, \cdots, $ unless some stopping criterion associated with the nonlinear system \eqref{eq:opti:u:lambda} is satisfied.}
		\STATE {\textbf{Step 1: }Unless some stopping criterion is satisfied, solve the linear system \eqref{eq:ssn:pdd:h} for $h^{l+1}$ for the isotropic case (or \eqref{eq:pdssn:h:ani} for the anisotropic case) with iterative method (BiCGSTAB):}
		\STATE{\textbf{Step 2: }Update $u^{l+1} = H^{-1}(f + \Div h^{l+1})$}
		\STATE{\textbf{Step 3: }Project $h^{l+1}$ to the feasible set  $\{h:\norm[\infty]{h} \leq \alpha\}$, i.e., $h^{l+1}  =\mathcal{P}_{\alpha}(h^{l+1})$. Set $(u^{l+1},h^{l+1})$ as the initial value for the next Newton iteration and go to \textbf{Step 1}}.
		\STATE{Output $u^{k+1}$, $h^{k+1}$}
	\end{algorithmic}
\end{algorithm} 
\begin{remark}
	The projection to the feasible set $\{h:\norm[\infty]{h} \leq \alpha\}$ is very important for the positive definiteness of \eqref{eq:newton:equation:iso} or \eqref{eq:newton:equation:ani}. 
	It can bring out more efficiency for solving the linear systems \eqref{eq:newton:equation:iso} or \eqref{eq:newton:equation:ani} numerically as in our numerical tests (see also \cite{HS}). 
\end{remark}
\section{ALM with Semismooth Newton Involving Soft Thresholding Operators}\label{sec:alm:ssnp:thres}
\subsection{ALM with Semismooth Newton: the Isotropic Case}

For the isotropic TV, we can rewrite the equation \eqref{eq:u:alm:suntoh1} as follows
\begin{equation}\label{eq:newton:iso:u:p}
F^I(u)=0, \quad F^I(u):=Hu-f + \nabla^* \lambda^k + \sigma_k \nabla^*\nabla u - \sigma_k \nabla^* S_{\frac{\alpha}{\sigma_k}}^{I}(\hat \lambda^k+\nabla u). 
\end{equation}
We now turn to the semismooth Newton method to solve \eqref{eq:newton:iso:u:p}. The main problem comes from the soft thresholding operator $ S_{\frac{\alpha}{\sigma_k}}^{I}(\cdot)$. Still, let's introduce  the active sets
\[
\chi_{u, \hat \lambda^k}^{+} = \begin{cases}
1, \ \  \ |\hat \lambda^k+ \nabla u | \geq {\alpha}/{\sigma_k}, \\
0, \ \ \  |\hat \lambda^k+ \nabla u | < {\alpha}/{\sigma_k}.
\end{cases} 
\]

With the Moreau's indentity \eqref{eq:moreau:sub}, with direct calculations, we have
\begin{equation}\label{eq:proj:thres:rela}
S_{\frac{\alpha}{\sigma_k}}^{I}(\hat \lambda^k+  \nabla u ) = \hat \lambda^k +  \nabla u- \mathcal{P}_{\frac{\alpha}{\sigma_k}}( \hat \lambda^k +  \nabla u).
\end{equation}\label{eq:proj:thres:rela:gradient}
By \cite{CL} (Corollary 2 in section 2.3.3), we have
\begin{equation}
\partial_u(S_{\frac{\alpha}{\sigma_k}}^{I}(\hat \lambda^k+ \nabla u ) ) = \nabla - \partial_u( \mathcal{P}_{\frac{\alpha}{\sigma_k}}( \hat \lambda^k +  \nabla u)).
\end{equation}
Denote $P(u) = \mathcal{P}_{\frac{\alpha}{\sigma_k}}( \hat \lambda^k +  \nabla u)$. 
Let's introduce  the active sets
\begin{equation}\label{eq:act:set:iso:kun:shreshold}
\chi_{u, \hat \lambda^k}^{+} = \begin{cases}
1, \ \  \ |\hat \lambda^k+ \nabla u | \geq {\alpha}/{\sigma_k}, \\
0, \ \ \  |\hat \lambda^k+ \nabla u | < {\alpha}/{\sigma_k},
\end{cases} \quad
\chi_{u, \hat \lambda^k}^{-} = \begin{cases}
1, \ \  \ |\hat \lambda^k+ \nabla u | < {\alpha}/{\sigma_k}, \\
0, \ \ \  |\hat \lambda^k+ \nabla u | \geq {\alpha}/{\sigma_k}.
\end{cases}
\end{equation}

Actually, we have the following lemma. 
\begin{lemma}\label{eq:proj:semismooth}
	$ P(u)=\mathcal{P}_{\frac{\alpha}{\sigma_k}}( \hat \lambda^k +  \nabla u)$  is a semismooth function of $u$. Furthermore, we have
	\begin{equation}\label{eq:inclusion:iso}
\left\{	A^{P,s}_u (\nabla \cdot \ )  \ | \ s \in [0,1] \right\} \subset 	(\partial_{u}  \mathcal{P}_{\frac{\alpha}{\sigma_k}}( \hat \lambda^k +  \nabla u)).
	\end{equation}
	It means that for any $v$ and  $s \in [0,1]$, we have
	\begin{equation} \label{eq:first:minus}
	A^{P,s}_u(\nabla \cdot \ ) = \begin{cases}
	A_{u}^{-}(\nabla \cdot \ ) :=    \nabla \cdot \ , &|\hat \lambda^k+ \nabla u | <{\alpha}/{\sigma_k}, \\
	\dfrac{\alpha}{\sigma_k}\left( \dfrac{\nabla \cdot \ }{|\hat \lambda^k+ \nabla u |} - s\dfrac{\langle \hat \lambda^k + \nabla u, \cdot \  \rangle (\hat \lambda^k + \nabla u) }{|\hat \lambda^k+ \nabla u |^3} \right),  \  &|\hat \lambda^k+ \nabla u | ={\alpha}/{\sigma_k},\\
 A_{u}^{+}(\nabla \cdot \ )=\dfrac{\alpha}{\sigma_k}\left( \dfrac{\nabla \cdot \ }{|\hat \lambda^k+ \nabla u |} - \dfrac{\langle \hat \lambda^k + \nabla u, \nabla \cdot \  \rangle (\hat \lambda^k + \nabla u) }{|\hat \lambda^k+ \nabla u |^3} \right), &|\hat \lambda^k+ \nabla u | > {\alpha}/{\sigma_k}.
	\end{cases}
	\end{equation}
	Throughout this paper, we choose $s=1$ for the Newton derivatives in \eqref{eq:inclusion:iso} and \eqref{eq:first:minus}, i.e.,
	\begin{equation}\label{eq:iso:proj}
	 A_{u,I}^{P} := \chi_{u, \hat \lambda^k}^{+}   A_{u}^{+}  + \chi_{u, \hat \lambda^k}^{-}   A_{u}^{-}.
	\end{equation}
\end{lemma}
\begin{proof}
	The semismoothness of $P(u)$ is as follows. It is known that  $L(y) = 	\mathcal{P}_{\frac{\alpha}{\sigma_k}}(y)$ is $PC^{\infty}$ function \cite{MU} (Example 5.16) (or see \cite{FP} Theorem 4.5.2 for more general projections
	 of $PC^1$ function). Since $y(u) =  \nabla u + \hat \lambda^k $ is differentiable and affine on $u$, $y(u)$ is also semismooth on $u$. We thus get the semismoothness of $P(u)=L(y(u))$ on $u$ \cite{MU} (Proposition 2.9).
	
	When $|\hat \lambda^k+ \nabla u | < {\alpha}/{\sigma_k}$, 
	\[
	P(u) = P_1(u) := \mathcal{P}_{\alpha}(\hat \lambda^k+  \nabla u ) = (\hat \lambda^k +  \nabla u).
	\]
Since	$P_1(u)$ is an affine and differentiable function on $u$, we have
	\[
	P_1(u) = \hat \lambda^k + \nabla u, \quad \nabla_u P_1(u) = \nabla.
	\] 
	The G\^ateaux derivative of $\mathcal{P}_{\frac{\alpha}{\sigma_k}}( \hat \lambda^k +  \nabla u)$ can be directly calculated as in \eqref{eq:first:minus}.
	We thus have the first case of \eqref{eq:first:minus}.
	
	While $|\hat \lambda^k+ \nabla u | > \frac{\alpha}{\sigma_k}$, $P(u)$ is a smooth function on $u$. Similarly, we have 
	\begin{align}
	P(u) &= P_2(u):= \mathcal{P}_{\frac{\alpha}{\sigma_k}}( \hat \lambda^k +  \nabla u) = \dfrac{\alpha}{\sigma_k}\dfrac{ (\hat \lambda^k+  \nabla u)}{ |\hat \lambda^k +  \nabla u|}, \  \ |\hat \lambda^k+ \nabla u | > {\alpha}/{\sigma_k},\\
	\nabla_uP_2(u) &= \dfrac{\alpha}{\sigma_k}\left[\dfrac{\nabla \cdot }{  |\hat \lambda^k +  \nabla u|} - \dfrac{ \langle  \hat \lambda^k  +  \nabla u, \nabla \cdot  \ \rangle (\hat \lambda^k+  \nabla u )}{|\hat \lambda^k +  \nabla u|^3}\right], \ \ |\hat \lambda^k+ \nabla u | > {\alpha}/{\sigma_k}.
	\end{align}
	Actually, it can be readily verified by the directional derivative as follows.
	\begin{align}
	&\partial_{u}\dfrac{ (\hat \lambda^k+  \nabla u)}{ |\hat \lambda^k +  \nabla u|}(v)
	= \lim_{t \rightarrow 0}\dfrac{1}{t}\left( \dfrac{\hat \lambda^k+  \nabla u + t\nabla v}{ |\hat \lambda^k +  \nabla u+t\nabla v| } - \dfrac{\hat \lambda^k +  \nabla u}{ |\hat \lambda^k +  \nabla u| }\right)\\
	&=\lim_{t \rightarrow 0}\dfrac{1}{t}\left( \dfrac{\hat \lambda^k+  \nabla u + t\nabla v}{ \sqrt{ |\hat \lambda^k +  \nabla u|^2 + 2t \langle  \hat \lambda^k +  \nabla u, \nabla v\rangle  +t^2 |\nabla v|^2} } - \dfrac{\hat \lambda^k +  \nabla u}{ |\hat \lambda^k  +  \nabla u| }\right) \notag \\
	&=\lim_{t \rightarrow 0}\dfrac{1}{t}\left( \dfrac{\hat \lambda^k+  \nabla u + t\nabla v}{  |\hat \lambda^k +  \nabla u|}\bigg(1-\dfrac{1}{2}\dfrac{2t \langle  \hat \lambda^k +  \nabla u, \nabla v\rangle  + t^2 |\nabla v|^2 }{\|\hat \lambda^k +  \nabla u\|^2} + \mathcal{O}(t^2)\bigg) - \dfrac{\hat \lambda^k+  \nabla u}{ |\hat \lambda^k +  \nabla u| }\right) \notag \\
	& = \lim_{t \rightarrow 0}\dfrac{1}{t}\left( \dfrac{t\nabla v}{  |\hat \lambda^k +  \nabla u|} - \dfrac{t \langle  \hat \lambda^k +  \nabla u, \nabla v \rangle (\hat \lambda^k+  \nabla u )}{|\hat \lambda^k +  \nabla u|^3}  +  \mathcal{O}(t^2) \right) \\
	& = \dfrac{\nabla v}{  |\hat \lambda^k +  \nabla u|} - \dfrac{ \langle  \hat \lambda^k  +  \nabla u, \nabla v \rangle (\hat \lambda^k+  \nabla u )}{|\hat \lambda^k +  \nabla u|^3}.
	\end{align}
	For any $u$ such that  $|\hat \lambda^k+ \nabla u | = {\alpha}/{\sigma_k} $, by \cite{Sch} (Proposition 4.3.1), we have 
	\[
	\text{co}\{ (\nabla_u P_1(u)), (\nabla_uP_2(u))\} = A^P_u (\nabla \cdot) \subset \partial_u P(u),
	\]
	which leads to \eqref{eq:inclusion:iso} and \eqref{eq:first:minus}. By \cite{CL} (Corollary 2.6.6), we have
	\[
	\nabla^* A^P_u(\nabla v) \subset (\nabla^* \partial_uP(u))(v) = \partial_u(\nabla^*P(u))(v).
	\]
\end{proof}

Similarly, we have the following lemma. 
\begin{lemma}\label{lem:positive:iso:p}
	$F^I(u)$ is a semismooth function of $u$. We have
	\begin{align}
	&A^S_u (\nabla v)\subset	(\partial_{u} S_{\frac{\alpha}{\sigma_k}}^{I}(\hat\lambda^k+\nabla u)(v), \label{eq:inclusion:iso:p} \\
	A^S_u(\nabla v ) &= \begin{cases}
	0, &|\hat \lambda^k+ \nabla u | < \frac{\alpha}{\sigma_k} \\
	\left(\nabla v - \dfrac{\alpha}{\sigma_k}\left( \dfrac{\nabla v}{|\hat \lambda^k+ \nabla u |} - s\dfrac{\langle \hat \lambda^k + \nabla u, \nabla v \rangle (\hat \lambda^k + \nabla u) }{|\hat \lambda^k+ \nabla u |^3} \right)\right), \ s \in [0,1], \  &|\hat \lambda^k+ \nabla u | = \frac{\alpha}{\sigma_k} \\
	A_{u}^{+}(\nabla v) :=\left(\nabla v - \dfrac{\alpha}{\sigma_k}\left( \dfrac{\nabla v}{|\hat \lambda^k+ \nabla u |} - \dfrac{\langle \hat \lambda^k + \nabla u, \nabla v \rangle (\hat \lambda^k + \nabla u) }{|\hat \lambda^k+ \nabla u |^3} \right)\right). &|\hat \lambda^k+ \nabla u | > \frac{\alpha}{\sigma_k}
	\end{cases}\label{eq:first:minus1p}
	\end{align}
	Throughout this paper, we choose the following generalized gradient for computations
	\[
	A_{u}^I(\nabla v) :=  \chi_{u, \hat \lambda^k}^{+} A_{u}^{+}(\nabla v) \subset  A^S_u(\nabla v ) ,
	\]
	where we always choose $s=1$ in \eqref{eq:first:minus1p}.
	The Newton derivative of \eqref{eq:newton:iso:u:p} can be chosen as 
	\begin{equation}\label{eq:sub:iso:primal}
	H +\sigma_k \nabla^*\nabla - \sigma_k\nabla ^* A_u^I \nabla,
	\end{equation}
	which is positive definite with lower bound on $u$ and  thus satisfies the regularity condition,
	\begin{equation}
	\langle (H +\sigma_k \nabla^*\nabla - \sigma_k\nabla ^* A_u^I \nabla)u, u\rangle_{X}  \geq \|u\|_{H}^2.
	\end{equation}
\end{lemma}
\begin{proof}
	By Lemma \ref{eq:proj:semismooth} and \eqref{eq:proj:thres:rela}, we see $S_{\frac{\alpha}{\sigma_k}}^{I}(\hat \lambda^k+  \nabla u ) $ is semismooth and the semismoothness of $F^I(u)$ then follows. Furthermore, by  \eqref{eq:proj:thres:rela:gradient}, we obtain \eqref{eq:inclusion:iso:p}. Since
	\[
	H +\sigma_k \nabla^*\nabla - \sigma_k\nabla ^* A_u^I \nabla \subset \partial F^I(u) = H +\sigma_k \nabla^*\nabla - \sigma_k \nabla^*  \partial_u(S_{\frac{\alpha}{\sigma_k}}^{I}(\frac{\lambda^k}{\sigma_k}+\nabla u)),
	\]
	we can choose $	I +\sigma_k \nabla^*\nabla - \sigma_k\nabla ^* A_u^I \nabla$ as a Newton derivative for $F^I(u)$. 

	For the positive definiteness of the Newton derivative \eqref{eq:sub:iso:primal}, we have
	\begin{align*}
	&\langle (H + \sigma_k \nabla^* \nabla -\nabla ^* A_u^I \nabla)u, u\rangle_{X} = \|u\|_{H}^2 +\sigma_k\langle \nabla u, \nabla u \rangle_{Y} -\sigma_k \langle A_u^I \nabla u, \nabla u \rangle_{Y} \\
	=&  \|u\|_{H}^2 + \sigma_k\langle  (I-\chi_{u, \hat \lambda}^{+ }) \nabla u, \nabla u \rangle_{Y} + \frac{\alpha}{\sigma_k} \langle  \chi_{u, \hat \lambda}^{+ }\dfrac{\nabla u}{  |\hat \lambda^k +  \nabla u|} -  \chi_{u, \hat \lambda}^{+ }\dfrac{ \langle  \hat \lambda^k  +  \nabla u, \nabla u \rangle (\hat \lambda^k+  \nabla u )}{|\hat \lambda^k +  \nabla u|^3}, \nabla u\rangle_{Y} \\
	=&  \|u\|_{H}^2 + \sigma_k\langle  (I-\chi_{u, \hat \lambda}^{+ }) \nabla u, \nabla u \rangle_{Y} + \frac{\alpha}{\sigma_k} \langle  \chi_{u, \hat \lambda}^{+ }\dfrac{\nabla u}{  |\hat \lambda^k +  \nabla u|} , \nabla u\rangle_{Y} \\
	&-\alpha \langle \chi_{u, \hat \lambda}^{+ }\dfrac{ \langle  \hat \lambda^k  +  \nabla u, \nabla u \rangle (\hat \lambda^k+  \nabla u )}{|\hat \lambda^k +  \nabla u|^3}, \nabla u\rangle_{Y}  
	\geq  \|u\|_{H}^2,
	\end{align*}
	since $\chi_{u, \hat \lambda}^{+ } \leq I$ and by the comparison the integral functions
	\begin{align*}
	&\langle \chi_{u, \hat \lambda}^{+ }\dfrac{ \langle  \hat \lambda^k  +  \nabla u, \nabla u \rangle (\hat \lambda^k+  \nabla u )}{|\hat \lambda^k +  \nabla u |^3}, \nabla u\rangle =  \chi_{u, \hat \lambda}^{+ }\dfrac{ \langle  \hat \lambda^k  +  \nabla u, \nabla u \rangle^2 }{|\hat \lambda^k +  \nabla u|^3} \\
	& \leq  \chi_{u, \hat \lambda}^{+ }\dfrac{ |  \hat \lambda^k  +  \nabla u|^2 |\nabla u|^2}{|\hat \lambda^k +  \nabla u|^3} = \chi_{u, \hat \lambda}^{+ }\dfrac{ |\nabla u|^2}{|\hat \lambda^k +  \nabla u|}.
	\end{align*}
\end{proof}
The semismooth Newton method for solving \eqref{eq:u:alm:suntoh1} follows
\begin{equation}\label{eq:ssn:PT:iso}
(H + \sigma_k \nabla ^* \nabla - \sigma_k \nabla^* A_{u^l}^I \nabla)\delta u^{l+1} = - F^I(u^l).
\end{equation}
\subsection{ALM with Semismooth Newton: the Anisotropic Case}
For the anisotropic $l_1$ norm \eqref{eq:l1:iso:pd},
we have 
\begin{equation}\label{eq:h:ani:newton}
p = (p_1,p_2)^T=S_{\frac{\alpha}{\sigma_k}}(\hat\lambda^k+\nabla u) = (S_{\frac{\alpha}{\sigma_k}}(\hat\lambda_1^k +\nabla_1 u),S_{\frac{\alpha}{\sigma_k}}(\hat\lambda_2^k +\nabla_2 u) )^T, 
\end{equation}
\begin{equation}\label{eq:p:shreshold}
p_i =S_{\frac{\alpha}{\sigma_k}}(\hat\lambda_i^k +\nabla_i u) = \begin{cases}
\hat \lambda_i^k + \nabla_i u - {\alpha}/{\sigma_k}, \quad  &\hat \lambda_i^k + \nabla_i u  >  {\alpha}/{\sigma_k}, \\
0,  \quad  &|\hat \lambda_i^k + \nabla_i u|  \leq   {\alpha}/{\sigma_k}, \\
\hat \lambda_i^k + \nabla_i u + {\alpha}/{\sigma_k}, \quad  & \hat \lambda_i^k + \nabla_i u  <  -{\alpha}/{\sigma_k}.
\end{cases} \quad i=1,2,
\end{equation}
With \eqref{eq:p:shreshold}, the equation \eqref{eq:u:alm:suntoh1}  becomes
\begin{equation}\label{eq:u:cancel:p}
F^A(u)=0, \quad F^A(u): = Hu - f + \nabla^* \lambda^k + \sigma_k \nabla^*\nabla u - \sigma_k \nabla^* 
(I + \frac{\alpha}{\sigma_k} \partial  \|\cdot\|_{1})^{-1}({\lambda^k}/{\sigma_k} +  \nabla u).
\end{equation}
The Newton derivative is of critical importance for the semismooth Newton method to solve the equation \eqref{eq:u:cancel:p}.
Let's first introduce
\[
\chi_{1, \hat \lambda^k}^{+,s} = \begin{cases}
1, \ \  \ |\hat \lambda_1^k + \nabla_1 u| > {\alpha}/{\sigma_k}, \\
s, \ \  \ |\hat \lambda_1^k  + \nabla_1 u| = {\alpha}/{\sigma_k},\\
0, \ \ \  |\hat \lambda_1^k  + \nabla_1 u| < {\alpha}/{\sigma_k},
\end{cases} \quad
\chi_{2, \hat \lambda^k}^{+,s} = \begin{cases}
1, \ \  \ |\hat \lambda_2^k  + \nabla_2 u| > {\alpha}/{\sigma_k}, \\
s, \ \  \ |\hat \lambda_2^k  + \nabla_2 u| = {\alpha}/{\sigma_k},\\
0, \ \ \  |\hat \lambda_2^k  + \nabla_2 u| < {\alpha}/{\sigma_k}.
\end{cases}
s \in [0,1].
\]
Introduce $\chi_{i, \hat \lambda^k}^{+} : = \chi_{1, \hat \lambda^k}^{+,s}$ for $s=1$ with $i=1,2$.
For the Newton derivative of equation \eqref{eq:u:cancel:p}, we have the following lemma.
\begin{lemma}\label{lem:positive:ani:p}
	The function $F^A(u)$ is semismooth on $u$. We  have
		\begin{equation}\label{eq:subdifffer:ani:threshold}
	\left\{\chi_{\lambda^k, u}^{i,s} \nabla_i \ | \ s\in [0,1] \right \} = \partial_{u}(S_{\frac{\alpha}{\sigma_k}}(\dfrac{\lambda_i^k}{\sigma_k} +\nabla_i u)), \quad i=1,2.
	\end{equation} 	
	The Newton derivative of the equation \eqref{eq:u:cancel:p} can be choose as
\begin{equation}\label{eq:gene:newton:deri:ani}
H + \sigma_k \nabla^*\nabla - \sigma_k \nabla^* A_u^A(\nabla),
\end{equation}
where $A_u^A(\nabla)$ can be chosen as the Newton derivative of $H(u)$	
\begin{equation}\label{eq:subdiff:ani:shreshold12}
A_u^A(\nabla) =\begin{bmatrix}
\chi_{1, \hat \lambda^k}^{+}  & 0 \\
0 & \chi_{2, \hat \lambda^k}^{+} 
\end{bmatrix}  \begin{bmatrix} \nabla_1 \\  \nabla_2\end{bmatrix}.
\end{equation}
Furthermore, the Newton derivative is positive definite and satisfies the regularity condition,
\begin{equation}
\langle (H + \sigma_k \nabla^*\nabla - \sigma_k \nabla^* A_u^A(\nabla))u, u\rangle_{X} \geq \|u\|_{H}^2.
\end{equation}
\end{lemma} 
\begin{proof}
%
It can be seen that $H_i(u) = S_{\frac{\alpha}{\sigma_k}}({\lambda_i^k}/{\sigma_k} +\nabla_i u)$ is piecewise differentiable, $i=1,2$. Take $H_1(u)$ for example. Denote $H_1^1(u) ={\lambda_1^k}/{\sigma_k} + \nabla_1 u - {\alpha}/{\sigma_k} $ and $H_1^2(u)=0$
 being  \emph{selection functions} of $H_1(u)$. For $|{\lambda_1^k}/{\sigma_k} + \nabla_1 u - {\alpha}/{\sigma_k}| >  {\alpha}/{\sigma_k}$,  
 \[ 
 H_1(u) = H_1^1(u), \quad \nabla_u H_1(u) = \nabla_u H_1^1(u) = \nabla_1,
 \]
 and while  $|{\lambda_1^k}/{\sigma_k} + \nabla_1 u - {\alpha}/{\sigma_k}| < {\alpha}/{\sigma_k}$,  
  \[ 
 H_1(u) = H_1^2(u), \quad \nabla_u H_1(u) = \nabla_u H_1^2(u)=0.
 \]
 Finally, while $|{\lambda_1^k}/{\sigma_k} + \nabla_1 u - {\alpha}/{\sigma_k}| = {\alpha}/{\sigma_k}$,
 \[
 \partial_u H_1(u) = \text{co} \{  \nabla_u H_1^1(u),  \nabla_u H_1^2(u)\} = \text{co}\{ 0, \nabla_1 \},
 \]
 which leads to \eqref{eq:subdifffer:ani:threshold}. By Lemma \ref{lem:vector:semismooth:newton}, we found that $A_k(u)$ in \eqref{eq:subdiff:ani:shreshold12} can be chosen as a Newton derivative of $H(u)$.
Still, since $\nabla^*$ is a linear, bounded and thus is a differentiable mapping,  the Newton derivative of $F^A(u)$ can be choose as \eqref{eq:gene:newton:deri:ani} \cite{KK} (Lemma 8.15). For the positive definiteness, with $\chi_{1, \hat \lambda^k}^{+} \leq I$ and $ \chi_{2, \hat \lambda^k}^{+} \leq I$, we see
\begin{align*}
\langle (H + \sigma_k \nabla^*(\nabla - A_u^A(u)))u, u\rangle_{X} &= \|u\|_{H}^2 + \sigma_k \left[\langle \nabla_1 u, (I-\chi_{1, \hat \lambda^k}^{+})\nabla_1 u \rangle_X +   \langle \nabla_2 u, (I-\chi_{2, \hat \lambda^k}^{+})\nabla_2 u \rangle_X\right] \\
& \geq  \|u\|_{H}^2. 
\end{align*}
\end{proof}
Then the semismooth Newton method for solving equation \eqref{eq:u:cancel:p} follows
\begin{equation}\label{eq:ssn:PT:ani}
(H + \sigma_k \nabla^*\nabla - \sigma_k\nabla^* A_{u^l}^A(\nabla))(\delta u^{l+1}) = -F^A(u^l).
\end{equation}
It can be checked that the linear operator in \eqref{eq:sub:iso:primal} is also self-adjoint. We thus can use the efficient conjugate gradient (CG) to solve \eqref{eq:ssn:PT:iso} and \eqref{eq:ssn:PT:ani}. 
We conclude this section by the following Algorithm \ref{alm:SSN_PT}, i.e., the semismooth Newton method for the primal problem involving with the soft thresholding operator (SSNPT) \eqref{eq:u:alm:suntoh1}, which is subproblem for the $k$th iteration of ALM applying to \eqref{eq:ROF}. The soft thresholding operators are frequently seen in compressed sensing and we refer \cite{LST} for the celebrated framework of semismooth Newton based ALM for compressed sensing, which gives us a lot of inspiration. 

\begin{algorithm}[h]
	\caption{Semismooth Newton method with solving primal problem involving the soft thresholding operator (SSNPT) for  \eqref{eq:u:alm:suntoh1}
		\label{alm:SSN_PT}}
	\begin{algorithmic}
	\STATE {Given corrupted image $f$,  multiplier $\lambda^k$,  initial value for Newton $\delta u^0=0$, step size $\sigma_k$ of ALM, choose $\mu \in (0,1/2)$, $\theta \in (0,1)$, $\eta^0 =1$. \\
	Iterate the following steps for $l=0, 1, \cdots, $ unless some stopping criterion associated with the nonlinear system \eqref{eq:u:alm:suntoh1} is satisfied.}
		\STATE {\textbf{Step 1}: Unless some stopping criterion is satisfied, solve the linear system \eqref{eq:ssn:PT:iso} for $\delta u^{k+1}$ of the isotropic case with iterative method (CG) (or solve \eqref{eq:ssn:PT:ani} of the anisotropic case with CG, i.e., Conjugate Gradient method)}
    	\STATE {\textbf{Step 2}: Do backtracking Armijo line search as follows: }
    	\State{Compute $\xi^l = (H + \sigma_k \nabla^*\nabla - \sigma_k\nabla^* A_{u^l}^X(u^l))(u^l) \in \partial_u \Phi(u; \lambda^k,\sigma_k )|_{u=u^l} $.   $A_{u^l}^X = A_{u^l}^I$ or $A_{u^l}^X = A_{u^l}^A$  and $\Phi(u; \lambda^k,\sigma_k ) $ is in \eqref{eq:augmented:lagrangian:only:u} depending on the isotropic or the anisotropic case.}
	    \WHILE{$$\Phi(u^l + \eta^j \delta u^{l+1};\lambda^k,\sigma_k ) >  \Phi(u^l; \lambda^k,\sigma_k ) + \mu \eta^j \langle \xi^l, \delta u^{l+1} \rangle, $$}
			\STATE {\quad \qquad \qquad \qquad \ \  $\eta^{j+1} = \eta^j \theta$,}
		\ENDWHILE
		\STATE{\textbf{Step 3}: Update $u^{l+1} = u^l + \eta^j \delta u^{l+1}$}
		\STATE{Output $u^{k+1}$}
	\end{algorithmic}
\end{algorithm}

\section{Convergence of the Augmented Lagrangian Methods and the Corresponding Semismooth Newton Methods}\label{sec:convergece:ssn:alm}
For the proposed semismooth Newton based ALM methods, we need the convergence both for the inner semismooth Newton iterations and the outer ALM iterations. Let's begin with the inner Newton iterations followed by the convergence of ALM. 
\subsection{Convergence of the semismooth Newton method}
\begin{theorem}[Superlinear Convergence \cite{KK}] \label{thm:convergence:SSN}
	Suppose $x^*$ is a solution to $F(x)=0$ and that $F$ is Newton differentiable at $x^*$ with Newton derivative $G$. If $G$ is nonsingular for all $x \in N(x^*)$ and $\{ \|G(x)^{-1}\| : x \in N(x^*)\}$ is bounded, then the Newton iteration
	\[
	x^{l+1} = x^l-G(x^l)^{-1}F(x^l), 
	\]
	converges superlinearly to $x^*$ provided that $|x^0-x^*|$ is sufficiently small.
\end{theorem}
For the semismooth Newton method, once the Newton derivative exists and is nonsingular at the solution point, we can employ semismooth Newton method \cite{KK} possibly with some globalization strategy \cite{DFC}. 
We now turn to the semismoothness for our cases, where each of the Newton derivative satisfies the regularity condition and is nonsingular (uniform regular) with a lower bound; see lemmas \ref{lem:positive:iso:pd}, \ref{lem:positive:ani:pd}, \ref{lem:positive:ani:p}, \ref{lem:positive:iso:p}. For the convergence of the semismooth Newton methods including Algorithm \ref{alm:SSN_PDP} and \ref{alm:SSN_PDD}, we refer to \cite{HS, SH} for the analysis of semismooth Newton under perturbations.
For the convergence of the semismooth Newton method in Algorithm \ref{alm:SSN_PT} with line search, we refer to \cite{LST} (Theorem 3.6) and \cite{ZST} (Theorems 3.4 and 3.5).

 Here, we follow the standard stopping criterion for the inexact augmented Lagrangian method \cite{Roc1, Roc2} and \cite{LST, ZZST}.
\begin{align}
& \Phi_k(u^{k+1},h^{k+1}) - \inf \Phi_k \leq \epsilon_k^2/2\sigma_k, \quad \sum_{k=0}^{\infty}\epsilon_k < \infty, \label{stop:a}        \tag{A} \\
& \Phi_k(u^{k+1},h^{k+1}) - \inf \Phi_k \leq \delta_k^2/2\sigma_k\|\lambda^{k+1}-\lambda^k\|^2, \quad \sum_{k=0}^{\infty}\delta_k < +\infty, \label{stop:b1} \tag{B1}\\
&\text{dist}(0, \partial \Phi_k(u^{k+1}, h^{k+1})) \leq \delta_k'/\sigma_k\|\lambda^{k+1} - \lambda^k\|, \quad 0 \leq \delta_k' \rightarrow 0. \label{stop:b2}\tag{B2}
\end{align}

We conclude this section with the following algorithmic framework of ALM with Algorithm \ref{alm:SSN_ALM}. Henceforth, we denote ALM-PDP or ALM-PDD as the ALM with the Algorithm \ref{alm:SSN_PDP} (SSNPDP) or Algorithm \ref{alm:SSN_PDD} (SSNPDD). We also denote  ALM-PT as the ALM with the Algorithm \ref{alm:SSN_PT} (SSNPT).
\begin{algorithm}[h]
	\caption{General framework of semismooth Newton based ALM for \eqref{eq:ROF}
		\label{alm:SSN_ALM}}
	\begin{algorithmic}
		\STATE {\textbf{Input}: corrupted image $f$, regularization parameter $\alpha$ for TV,  initial Lagrangian multiplier $\lambda^0$,  initial step size $\sigma_0$ and the largest step size $\sigma_{\infty}$ of ALM, step size update parameter $c>1$}
		\STATE {\textbf{Iteration}: Iterate the following steps for $k=0, 1, \cdots, $ unless some stopping criterion associated with the problem \eqref{eq:ROF} is satisfied.}	
\STATE {\textbf{Step 1}: With $\sigma_k$, $\lambda^k$, solve the minimization problem \eqref{eq:alm:up} with $u^k$  (possibly $h^k$ for Algorithm  \ref{alm:SSN_PDP} or Algorithm \ref{alm:SSN_PDD})  by \\
 	 \qquad \qquad Algorithm \ref{alm:SSN_PDP} or Algorithm \ref{alm:SSN_PDD} for the nonlinear system \eqref{eq:opti:u:lambda} or \eqref{eq:opti:u:lambda:ani} for $(u^{k+1}, h^{k+1})$; \\ 
	 \qquad \qquad  or Algorithm \ref{alm:SSN_PT} for the nonlinear system \eqref{eq:u:alm:suntoh1} for $u^{k+1}$.}
\STATE {\textbf{Step 2}: Update Lagrangian multiplier $\lambda^{k+1}$: }
\State{\qquad \qquad For Algorithm \ref{alm:SSN_PT} using \eqref{eq:update:lambda}, i.e.,   $$\lambda^{k+1} = \lambda^k + \sigma_k(\nabla u^{k+1} -p^{k+1}), \quad \text{with} \ \  p^{k+1}  = S_{\frac{\alpha}{\sigma_k}}(\frac{\lambda^k}{\sigma_k} + \nabla u^{k+1}).$$}
\qquad \qquad For Algorithm \ref{alm:SSN_PDP} or \ref{alm:SSN_PDD} using \eqref{eq:update:multiplier:projection}, i.e., 
$$\lambda^{k+1} = \mathcal{P}_{\alpha}(\lambda^k + \sigma_k \nabla u^{k+1}).$$
\STATE{\textbf{Step 3}: Update step size  $\sigma_{k+1} = c \sigma_k \leq \sigma_{\infty}$.}
	\end{algorithmic}
\end{algorithm} 

\subsection{Convergence of the Augmented Lagrangian Method}

It is well-known that the  augmented Lagrangian method  can be seen as applying proximal point algorithm to the dual problem \cite{Roc1,Roc2}. The convergence and the corresponding rate of augmented Lagrangian method are closely related to the convergence of the proximal point algorithm. Especially, the local linear convergence of the multipliers, primal or dual variables is mainly determined by the metric subregularities of the corresponding maximal monotone operators \cite{Roc1,Roc2, LE, LU}. We now turn to analysis the metric subregularity of the corresponding maximal monotone operators which is usually efficient for the asymptotic (or local) linear convergence of ALM.
 
Now we introduce some basic definitions and properties of multivalued mapping from convex analysis \cite{DR, LST}. Let $F: X \Longrightarrow Y $ be a multivalued mapping. The graph of $F$ is defined as the set
\[
\text{gph} F: = \{ (x,y) \in X\times Y| y\in F(x)\}.
\]
The inverse of $F$, i.e., $F^{-1}:  Y \Longrightarrow X$ is defined as the multivalued mapping whose graph is $\{(y,x)| (x,y) \in \text{gph} F\}$. The distance $x$ from the set $C\subset X$ is defined by
\[
\text{dist}(x,C): = \inf\{\|x-x'\|\ | \ x' \in C\}.
\]
Let's introduce the metrical subregularity for $F$ and the calmness \cite{DR, LST}.
\begin{definition}[Metric Subregularity \cite{DR}]\label{def:metricregular}
	A mapping $F: X \Longrightarrow Y$ is called metrically subregular at $\bar x$ for $\bar y $ if $(\bar x, \bar y) \in \text{gph} F$ and there exists modulus $\kappa \geq 0$  along with a neighborhoods $U$ of $\bar x$ and $V$ of $\bar y$ such that
	\begin{equation}\label{eq:metricregular}
	\text{dist}(x, F^{-1}(\bar y)) \leq \kappa  \text{dist}(\bar y, F(x) \cap V ) \quad \text{for all} \ \ x \in U.
	\end{equation}
\end{definition}
\begin{definition}[Calmness \cite{DR}]A mapping $S: \mathbb{R}^m \rightrightarrows \mathbb{R}^n$ is called calm at $\bar y$ for $\bar x$ if $(\bar y, \bar x) \in \text{gph} \ S$, and there is a constant $\kappa \geq 0$ along with  neighborhoods $U$ of $\bar x$ and $V $ of  $\bar y$ such that 
	\begin{equation}\label{calmness:def}
	S(y) \cap U \subset S(\bar y) + \kappa |y-\bar y| \mathbb{B}, \quad \forall y \in V.
	\end{equation}
	In \eqref{calmness:def}, $\mathbb{B}$ denotes the closed unit ball in $\mathbb{R}^n$.
\end{definition}

Let's now vectorize  the variables disscussed in detail. Suppose 
\begin{align*}
& \nabla_1 \in \mathbb{R}^{m\times n}: \mathbb{R}^n \rightarrow \mathbb{R}^m, \ \  \nabla_2\in \mathbb{R}^{m\times n}: \mathbb{R}^n \rightarrow \mathbb{R}^m, \ \ u=(u^1, \cdots, u^n)^{T} \in \mathbb{R}^n, \\
&p=(p_1,\cdots, p_m)^T \in \mathbb{R}^{2m}, \quad p_i = (p_i^1, p_i^2)^T \in \mathbb{R}^2, \\
& \lambda=(\lambda_1,\cdots, \lambda_m)^T \in \mathbb{R}^{2m}, \quad \lambda_i = (\lambda_i^1, \lambda_i^2)^T \in \mathbb{R}^2. 
\end{align*}
For the anisotropic case, we see
\[
\|p\|_{1} =\sum_{i=1}^m \|p_i\|_1 = \sum_{i=1}^m|p_i^1| + \sum_{i=1}^m|p_i^2|=\sum_{i=1}^m(|p_i^1| + |p_i^2|),
\]
which is a polyhedral function. For the istropic case, we notice 
\[
\|p\|_{1} = \sum_{i=1}^m |p_i| = \sum_{i=1}^m \sqrt{{p_i^1}^2 +{p_i^2}^2}
\]
which is not a polyhedral function. Fortunately, it is a group Lasso norm \cite{YY,ZZST}. 
Now, let's turn to the anisotropic case first. 
Introduce the Lagrangian function
\begin{equation}
l(u,p,\lambda) = \frac{1}{2}\langle Qu, u \rangle_{2} - \langle u, b \rangle   + \frac{1}{2}\|z\|_{2}^2 + \alpha \|p\|_{1} + \langle \nabla u-p, \lambda \rangle,
\end{equation}
where $Q$ is a linear, positive semidefinite operator and $b$ is known. For \eqref{eq:ROF}, we have $Q = H=-\mu \Delta + K^*K$ and $b=K^*z$.
It is well-known that $l$ is a convex-concave function on $(u,p,\lambda)$. Define the maximal monotone operator $T_{l}$ by 
\begin{equation}
T_{l}(u,p,\lambda) =\{(u',p',\lambda')|(u',p',-\lambda')\in \partial l(u,p,\lambda)\},
\end{equation}z
and the corresponding inverse is given by
\begin{equation}
T_{l}^{-1}(u',p',\lambda') =\{(u,p,\lambda)|(u',p',-\lambda')\in \partial l(u,p,\lambda)\}.
\end{equation}
\begin{theorem}\label{thm:metric:regular:lag}
	For the anisotropic ROF model, assuming the KKT system has at least one solution, then $T_{l}$ is metrically subregular at $(\bar u, \bar p, \bar \lambda)^T$ for the origin. 
\end{theorem}
\begin{proof}
Let's consider the general case including the case in \eqref{eq:ROF}. Actually, we have 
\[
T_{l}(u,p,\lambda) = (Qu-b, -\lambda + \alpha \partial \|p\|_{1}, p-\nabla u)^{T}: = \mathcal{A}(x) + \mathcal{B}(x),
\]
where
\begin{equation}
\mathcal{A}\begin{pmatrix}
u\\p\\ \lambda
\end{pmatrix}
:=\begin{pmatrix}
0 & 0& 0\\
0 & \alpha \partial \|\cdot\|_{1} &0 \\
0 &0&0
\end{pmatrix}\begin{pmatrix}
u\\p\\ \lambda
\end{pmatrix},
\quad 
\mathcal{B}\begin{pmatrix}
u\\p\\ \lambda
\end{pmatrix} := 
 \begin{pmatrix}
Q & 0 &\nabla^* \\
0&0&-I \\
-\nabla & I &0
\end{pmatrix}
\begin{pmatrix}
u\\p\\ \lambda
\end{pmatrix}
+\begin{pmatrix}
-b \\ 0 \\0
\end{pmatrix}.
\end{equation}
It can be seen that the monotone operator $\mathcal{A}$ is polyhedral since the anisotropic $\|\cdot\|_{1}$ \eqref{eq:l1:iso:pd} is a polyhedral convex function and the operator $\mathcal{B}$ is a maximal monotone and affine operator. Thus $T_{l}$ is a polyhedral mapping \cite{Rob}. By the corollary in \cite{Rob}, we see $T_{l}$ is metrically subregular at $(\bar u, \bar p, \bar \lambda)^{T}$ for the origin.
\end{proof}
Let's now turn to the metric subregularity of $\partial d$ for the dual problem \eqref{eq:dual:rof}, supposing $(\partial d)^{-1}(0) \neq \emptyset$ and there exists $\bar \lambda$ such that $0 \in (\partial d)(\bar \lambda)$,
\begin{equation}\label{eq:subgradient:dual}
(\partial d)(\lambda) = \Div^*H^{-1}(\Div \lambda + f) + \partial g(\lambda), \quad g(\lambda): = \mI_{\{\norm[\infty]{\lambda} \leq \alpha\}}(\lambda).
\end{equation}

For the anisotropic case, actually, the constraint set is a polyhedral convex set in $\mathbb{R}^{2m}$, since
\[
g(\lambda) = 0 \Leftrightarrow \left\{\lambda=(\lambda_1,\lambda_2, \cdots, \lambda_m)^{T} \ | \ \lambda_i \in \mathbb{R}^{2}, \  |\lambda_i^k|  \leq \alpha, \ i =1, 2, \cdots, m, \ k=1,2\right\}.
\]
Together with $\Div^*H^{-1}(\Div \lambda + f) $ being an affine and monotone mapping, $\partial g$ is a polyhedral mapping by \cite{Rob}.
This leads to that $\partial g$ is metrically subregular at $\bar \lambda$ for the origin with similar argument as in Theorem \ref{thm:metric:regular:lag}.  

Now we turn to the isotropic case. The metric subregularity of $\partial d$ is more subtle, since the constraint set
\[
g(\lambda) = 0 \Leftrightarrow \left\{\lambda=(\lambda_1, \cdots, \lambda_m)^{T} \ | \ \lambda_i \in \mathbb{R}^{2}, \ |\lambda_i|= \sqrt{ (\lambda_i^1)^2 + (\lambda_i^2)^2} \leq \alpha, \ i =1,\cdots, m\right\}
\]
is not a polyhedral set. Now introduce
\[
g_i(\lambda_i) = \mI_{\{ |\lambda_i|\leq \alpha\}}(\lambda_i), \quad i=1,2,\cdots, m.
\]
Henceforth, let's denote $\mathbb{B}_{a}(\bar \lambda_i)$ or $\mathbb{B}^k_{a}(\bar \lambda_i)$ as the Euclidean closed ball centered at $\bar \lambda_i \in \mathbb{R}^2$ with radius $a$, i.e., 
\begin{equation}\label{eq:l2:ball}
\mathbb{B}_{\alpha}^i(0) :=  \left\{\lambda_{i}: = (\lambda_{i}^1,\lambda_{i}^2)^{T} \in \mathbb{R}^2 \ | \ |\lambda_{i}| = \sqrt{(\lambda_{i}^1)^2 + (\lambda_{i}^2)^2} \leq \alpha\right\}, \ \  i = 1,\cdots, m, \ \alpha >0.
\end{equation}
Furthermore, denote $ \mathbb{B}_{a}( \lambda) = \Pi_{i=1}^m \mathbb{B}_{a}(\bar  \lambda_i)$ with $\bar \lambda = (\bar \lambda_1, \cdots, \bar \lambda_m)^T$. We can thus write
\[
\partial g = \Pi_{i=1}^m\partial g_{i} = \Pi_{i=1}^m\partial  \mI_{\mathbb{B}_{\alpha}^i(0)}(\lambda_i).
\]
It is known that each $\partial  \mI_{\mathbb{B}_{\alpha}^i(0)}(\lambda_i)$ is metrically subregular at $(\bar \lambda_i, \bar v_i) \in \text{gph} \partial  \mI_{\mathbb{B}_{\alpha}^i(0)} $ \cite{YY} (see Lemma 6 therein). For the metric subregularity of $\partial g$, we have the following lemma. 
\begin{lemma}\label{lem:metric:subregular:g}
	For any $(\bar \lambda, \bar v)^T \in \emph{gph} \ \partial g$, $\partial g$ is metrically  subregular at $\bar \lambda$ for $\bar v$.
\end{lemma}
\begin{proof}
	For any $(\bar \lambda, \bar v)^T \in \text{gph} \ \partial g$, and $V$ of a neignborhoods of $\bar \lambda$, since
	\begin{align*}
	&\text{dist}^2(\lambda, (\partial g)^{-1}(\bar v) ) = \sum_{i=1}^m 	\text{dist}^2(\lambda_i, (\partial g_i)^{-1}(\bar v_i) ) \\
	& \leq \sum_{i=1}^m \kappa_i^2 \text{dist}^2(\bar v_i, (\partial g_i)(\bar \lambda_i))	
	\leq \sum_{i=1}^m \max(\kappa_i^2, i=1,\cdots, m) \text{dist}^2(\bar v_i, (\partial g_i)(\bar \lambda_i))	\\
	&  = \max(\kappa_i^2, i=1,\cdots, m) \text{dist}^2(\bar v, (\partial g)(\bar \lambda)).
	\end{align*}
	Thus with choice $\kappa = \sqrt{ \max_{i=1}^m(\kappa_i^2, i=1,\cdots, m) }$, we found that  $\partial g$ is metrically  subregular at $\bar \lambda$ for $\bar v$ with modulus $\kappa$.
\end{proof}
By \cite{HS} (Theorem 2.1) (or formula (3.4) of \cite{KK1}), the solution $\bar u$ of the primal problem \eqref{eq:ROF} and the solution $\bar \lambda$ of the dual problem \eqref{eq:dual:rof} have the following relations
\begin{align}
H\bar u - \Div \bar \lambda &= f, \label{eq:hk:resi:eq0:u}\\
-\alpha \nabla \bar u + |\nabla \bar u|\bar \lambda &= 0, \quad |\bar \lambda| = \alpha, \label{eq:hk:resi:neq0} \\
 \nabla \bar u& = 0, \quad |\bar \lambda| < \alpha, \label{eq:hk:resi:eq0}
\end{align}
which can be derived from the optimality conditions \eqref{eq:opti:primaldual}. Now we turn to a more general model compared to \eqref{eq:dual:rof}.
Suppose $A \in \mathbb{R}^{m\times 2m}: \mathbb{R}^{2m} \rightarrow \mathbb{R}^m$, 
\begin{equation}\label{eq:dual:ROF:general}
f(\lambda) = f_1(\lambda) + \Pi_{i=1}^m\mI_{\mathbb{B}_{\alpha}^i(0)}(\lambda_{i}), \quad  f_1(\lambda):=\frac{\|  A \lambda -b\|^2 }{2} +  \langle q, \lambda \rangle + a_0
,\quad b\in \mathbb{R}^m,
\end{equation}
where $\mI_{\mathbb{B}_{\alpha}^i(0)}(x)$ is the indicator function for the  $l_2$ ball in \eqref{eq:l2:ball} and $a_0$ is a constant.
We claim that the model \eqref{eq:dual:ROF:general} also covers the dual problem \eqref{eq:dual:rof} by setting $A = H^{-\frac{1}{2}}\Div$, $b= -H^{-\frac{1}{2}}f$ and $a_0 = \frac{1}{2}\|z\|_{2}^2$. For the existence of the square root $H^{1/2}$ of the positive definite linear operator $H$, we refer to \cite{RS} (Theorem VI.9).
Supposing $g(\lambda) := \Pi_{i=1}^m\mI_{\mathbb{B}_{\alpha}^i(0)}(\lambda_{i})$, let's introduce
\begin{align*}
&\mathcal{X}: = \{ \lambda \ | \ A\lambda=\bar y, \quad -\bar g \in \partial g(\lambda)\}, \\
&\Gamma_1(p^1) = \{ \lambda \ | \ A \lambda - \bar y = p^1  \}, \quad 
\Gamma_2(p^2) = \{ \lambda \ | \ p^2 \in \bar g + \partial g(\lambda) \}, \\
&\hat \Gamma(p^1) =  \Gamma_1(p^1)\cap\Gamma_2(0) = \{ \lambda\ | \  p^1 = A\lambda - \bar y, \ 0 \in \bar g + \partial g(\lambda) \},
\end{align*}
where $\mathcal{X}$ is actually the solution set of \eqref{eq:dual:ROF:general}, since
\[
\bar g := A^T \nabla h(\bar y) +q=(\bar g_1, \bar g_2, \cdots, \bar g_m)^{T},\  h(y) = \|y-b\|^2/2, \quad \bar g_i \in \mathbb{R}^2.
\]
We also need another two set valued mapping,
\begin{align}
&\Gamma(p^1, p^2):=\{\lambda \ |\ p^1 = A \lambda -\bar y, \quad p^2 \in \bar g + \partial g(\lambda) \},\\
&S(p):=\{\lambda \  | \ p \in \nabla f_1(\lambda)+ \partial g(\lambda) \} \Rightarrow \mathcal{X}=S(0).
\end{align}
Actually the metric subregularity of $\partial f$ at $(\bar \lambda, 0)$ is equivalent to the calmness $S$ at $(0, \bar \lambda)$ \cite{DR}. Now we turn to the calmness of $S$. By \cite{YY} (Theorem 5), the calmness of $S$ at $(0, \bar \lambda)$ is equivalent to the calmness of $\hat \Gamma$ at $(0, \bar \lambda)$ for any $\bar \lambda \in S(0)$ if $\partial g$ is metrically subregular for $f$ in \eqref{eq:dual:ROF:general}. We would use the following calm intersection theorem to prove the calmness of $\Gamma$. 
\begin{proposition}[Calm intersection theorem \cite{KK1, KKU, YY}]\label{prop:calm:inter}
	Let $T_1: \mathbb{R}^{q_1} \rightrightarrows \mathbb{R}^n$, $T_2: \mathbb{R}^{q_2} \rightrightarrows \mathbb{R}^n$ be two set-valued maps. Define set-valued maps
	\begin{align}
	T(p^1, p^2): &= T_1(p^1)\cap T_2(p^2), \\
	\hat T(p^1):&= T_1(p^1)\cap T_2(0).
	\end{align}
	Let $\tilde x \in T(0,0)$. Suppose  both the set-valued maps $T_1$ and $T_2$ are calm at $(0, \tilde x)$ and $T_1^{-1}$ is pseudo-Lipschitiz at $(0,\tilde x)$. Then $T$ is calm at $(0,0, \tilde x)$ if and only if $\hat T$ is calm at $(0, \tilde x)$.
\end{proposition}
We need the following assumption first, which is actually a mild condition on the solution set $\mathcal{X}$ of the dual problem \eqref{eq:dual:ROF:general} by \eqref{eq:hk:resi:neq0}  and \eqref{eq:hk:resi:eq0}. 
\begin{assumption}\label{asump:existence}
With the normal cone $\mN_{\mathbb{B}_{\alpha}^i(0)}(\bar \lambda_i) = \partial  \mI_{\mathbb{B}_{\alpha}^i(0)}(\bar \lambda_i)$, let's assume that $\bar  \lambda \in \mathcal{X}$ and
\begin{itemize}
	\item[\emph{i.}] Either $\bar \lambda_{i}  \in \emph{bd} \mathbb{B}_{\alpha}^i(0)$ and there exists $\bar g_i \neq 0  $ such that  $-\bar g_i   \in \mN_{\mathbb{B}_{\alpha}^i(0)}(\bar \lambda_{i})$,
	\item[\emph{ii.}]  Either $\bar \lambda_{i}  \in \emph{int} \mathbb{B}_{\alpha}^i(0)$.
\end{itemize}	
\end{assumption}

\begin{theorem}\label{thm:metric:regular:dual:iso}
For the dual problem \eqref{eq:dual:ROF:general},  supposing it has at least one solution $\bar \lambda$ satisfying the Assumption \ref{asump:existence}, then $\partial f$ is metrically subregular at $\bar \lambda$ for the origin. 
\end{theorem}
\begin{proof}
	We mainly need to prove the calmness of $\hat\Gamma(p^1)$ at $(0,\bar \lambda)$. By metric subregularity of $\partial g$ by Lemma \ref{lem:metric:subregular:g} and the fact that $\Gamma_1^{-1}$ is pseudo-Lipschitiz and bounded metrically subregular at $(0,\bar \lambda)$ (\cite{YY}, Lemma 3), with the Calm intersection theorem as in Proposition \ref{prop:calm:inter},  we get calmness of $\Gamma$ at $(0,0, \bar \lambda)$. We thus get the calmness of $S$ at $(0, \bar \lambda)$ and the metric subregular of $\partial f$ at $\bar \lambda$ for the origin.
	
	 Now let's focus the the calmness of $\hat\Gamma(p^1)$ at $(0,\bar \lambda)$.
 Without loss of generality, suppose
\begin{align*}
& \bar \lambda_{i} \in \text{int}  \mathbb{B}_{\alpha}^i(0), \quad i = 1, \cdots, L; \\
& \bar \lambda_{i} \in \text{bd}  \mathbb{B}_{\alpha}^i(0), \quad  -\bar g_i  \neq 0 \in \mN_{\mathbb{B}_{\alpha}^i(0)}(\bar \lambda_i), \quad i = L+1, \cdots, m, \quad 1<L<m.
\end{align*}
The existence of $\bar g_i \neq 0$ is guaranteed by the Assumption \ref{asump:existence}.
For $i=1,\cdots, L$, $\bar \lambda_{i} \in \text{int} \mathbb{B}_{\alpha}^i(0)$, we have $\bar g_i \in \mN_{\mathbb{B}_{\alpha}^i(0)}(\bar \lambda_{i}) = \{0\}$. We thus conclude $\bar g_i =  0$ and
\begin{equation}\label{eq:gamma20:1}
\Gamma_2^i(0) = \{\lambda_i | 0 \in\mN_{\mathbb{B}_{\alpha}^i(0)}(\lambda_i) \} =  \mathbb{ B}_{\alpha}^i(0), \quad i = 1, \cdots, L.
\end{equation}
Now, let's turn to $\Gamma_2^i(0)$ for  $i=L+1,\cdots, m$. While $\bar \lambda_{i} \in \text{bd}\mathbb{ B}_{\alpha}^i(0)$, by the Assumption \ref{asump:existence} on the solution  $\bar \lambda$, there exists $\bar g_i \neq 0$ such that $-\bar g_i   \in \mN_{\mathbb{B}_{\alpha}^i(0)}(\bar \lambda_{i})$. 
Acutally,  we know that the normal cone at the boundary point of a closed ball is
\[ \quad\mN_{\mathbb{B}_{\alpha}^i(0)}( \lambda_i) = \{s \lambda_i|s\geq 0\}, \quad \forall \lambda_i \in \text{bd}\mathbb{ B}_{\alpha}^i(0).
\]
Still with the definition of $\Gamma_2^i(0)$, we have
\[
\Gamma_2^i(0) = \{\lambda_i | -\bar g_i \in\mN_{\mathbb{B}_{\alpha}^i(0)}(\lambda_i) \}. 
\]
While  $\lambda_i \in \text{int}\mathbb{ B}_{\alpha}^i(0)$, by the definition of $\Gamma_2(0)$, together with $\mN_{\mathbb{B}_{\alpha}^i(0)}(\lambda_i) = \{0\}$,  we see
$\bar g_i = 0$, which is contracted with the assumption \ref{asump:existence}. While  $\lambda_i \in \text{bd}\mathbb{ B}_{\alpha}^i(0)$, 
together with the definition of $\Gamma_2$,  we see $\lambda_i = -\bar g_i \frac{\alpha}{|\bar g_i|}$ with some $\frac{\alpha}{|\bar g_i|}>0$. Since $\bar g_i$ is fixed as in definition of  $\Gamma_2^i$ for $i=L+1, \cdots, m$ and $-\bar g_i \in \mN_{\mathbb{B}_{\alpha}^i(0)}(\bar\lambda_i)$ and $\bar \lambda_i \in  \text{bd}\mathbb{ B}_{\alpha}^i(0)$, we have $\lambda_i =\bar \lambda_i=-\bar g_i \frac{\alpha}{|\bar g_i|}$ and 
the only choice is 
\begin{equation}\label{eq:gamma20:2}
\Gamma_2^i(0) = \{ \bar  \lambda_{i}\} , \quad i = L+1, \cdots, m.
\end{equation}
Choose $\epsilon >0$ small enough such that $\mathbb{B}_{4\epsilon}^i(\bar \lambda_{i}) \subset \mathbb{B}_{\alpha}(0)$, $i=1,\cdots, L$. We thus conclude that
\begin{equation}\label{eq:cons:normacone}
\Gamma_2(0) \cap \mathbb{B}_{\epsilon}(\bar \lambda) = (\mathbb{B}_{\epsilon}^1(\bar \lambda_{1}), \cdots, \mathbb{B}_{\epsilon}^L(\bar \lambda_{L}), \bar  \lambda_{{L+1}}, \cdots, \bar  \lambda_{m})^{T}.
\end{equation}
 Suppose $p=(p_1, p_2, \cdots, p_m)^{T}$ and $ \lambda \in \Gamma_1(p) \cap \Gamma_2(0) \cap \mathbb{B}_{\epsilon}(\bar \lambda)$,  $p_i \in \mathbb{R}^2 $, $i=1, 2, \cdots, m$.
 Introduce the following constraint on $\lambda$ 
 \[
 \mathcal{R}: = \{ \lambda \ | \ \lambda_{i} \in \mathbb{R}^2, \ i=1, \cdots, L,  \   \lambda_{i} = \bar \lambda_{i}, \ i=L+1, \cdots, m   \},
 \]
 which is a convex and closed polyhedral set. It can be seen as follows. Denote $\bar L  = m-L$ and  $0_{2L\times2m} \in \mathbb{R}^{2L\times2m}$, $0_{2 \bar L\times2L} \in \mathbb{R}^{2 \bar L\times2L}$ as the zero matrix with elements are all zero. Denote $I_{2\bar L\times 2\bar L} \in \mathbb{R}^{2 \bar L\times2 \bar L}$ as the identity matrix. Introduce 
 \[
 E_{+} = [0_{2L\times2m};
 0_{2 \bar L\times2L}    \ I_{2 \bar L\times 2 \bar L}] \in \mathbb{R}^{2m \times 2m}, \quad 
  E_{-} = [0_{2L\times2m};
 0_{2 \bar L\times2L}    \ -I_{2 \bar L\times 2 \bar L}] \in \mathbb{R}^{2m \times 2m},
 \]
 \[
 E = [E_{+}; E_{-}] \in \mathbb{R}^{4m \times 2m}, \quad \bar \lambda_0 = [0, \cdots, 0, \bar \lambda_{L+1}, \cdots, \bar \lambda_{m}] \in \mathbb{R}^{2m}.
 \]
The set $\mathcal{R} = \{ \lambda \ | \ E \lambda \leq E \bar \lambda_0 \}$ is thus a polyhedral set. Actually, the following set 
 \begin{equation}\label{eq:matrix:equality}
M(p): =  \{  \lambda \ | \    A  \lambda -\bar y= p, \quad  \lambda \in \mathcal{R}\} = \{ \lambda \ | \ A \lambda -\bar y = p, \ E  \lambda \leq E \bar \lambda_0 \}, \ \ 
 \end{equation}
is also a polyhedral set. 

For any $\lambda  \in \Gamma_1(p)\cap\Gamma_2(0) \cap \mathbb{B}_{\epsilon}(\bar \lambda)=\hat \Gamma(p) \cap \mathbb{B}_{\epsilon}(\bar \lambda) $, denote $\tilde \lambda$ as its projection on $ M(0)$. Since $\bar \lambda \in M(0)$, we thus have
\[
\|\lambda - \tilde \lambda \| \leq \| \lambda - \bar \lambda\| \leq \epsilon  \Rightarrow \tilde \lambda \in \mathbb{B}_{\epsilon}( \lambda)   \Rightarrow  \tilde \lambda \in \mathbb{B}_{\alpha}(0).
\]
Together with $\tilde \lambda \in M(0)$ and $\tilde \lambda \in \mathcal{R}$, we see $\tilde \lambda \in \Gamma_2(0)$  by \eqref{eq:gamma20:1} and \eqref{eq:gamma20:2}. We thus conclude that $\tilde \lambda \in \hat \Gamma(0) =\Gamma_1(0)\cap\Gamma_2(0)$.    By the celebrated results of Hoffman error bound \cite{Hof} based on the polyhedral set in \eqref{eq:matrix:equality}, for any  $\lambda \in\hat \Gamma(p) \cap \mathbb{B}_{\epsilon}(\bar \lambda)$, there exists a constant $\kappa$ such that
 \begin{equation}
 \text{dist}(\lambda, \hat \Gamma(0)) \leq  \|\lambda - \tilde \lambda \| = \text{dist}(\lambda, M(0)) \leq \kappa \|p\|,\quad \forall \lambda  \in \hat \Gamma(p) \cap \mathbb{B}_{\epsilon}(\bar \lambda),
 \end{equation}
 since $E  \lambda \leq E \bar \lambda_0$ by $\lambda  \in \hat \Gamma(p)= \Gamma_1(p)\cap\Gamma_2(0)$.
We thus get the calmness of $\hat\Gamma(p)$ at $(0,\bar \lambda)$. While $L=0$, i.e., $\bar \lambda_{i} \in \text{bd}  \mathbb{B}_{\alpha}^i(0)$, $i=1, \cdots, m$, one can readily check that
$\hat \Gamma(p) = \Gamma_1(p)\cap\Gamma_2(0) =  \emptyset$ whenever $p_i \neq 0$ with $p=(p_1, \cdots, p_m)^{T}$. The calmness follows by definition and the proof is thus finished.
 \end{proof}

\begin{remark}
	Similar result is also given in \cite{YTP} (see Example 4.1(ii) of \cite{YTP}), where the delicate analysis based on LMI-representable (Linear Matrix Inequalities representable) functions are employed. It is proved that while the solution $\bar \lambda \in \mathcal{X}$ satisfies $0 \in \text{ri}~ \partial f(\bar \lambda)$, then $f$ has the Kurdyka-\L ojasiewicz exponent $1/2$ where $``\text{ri}" $ denotes the relative interior. 
\end{remark}

\begin{remark}
	For the the Assumption \ref{asump:existence} for the model \eqref{eq:dual:rof}, with the optimality conditions \eqref{eq:subgradient:dual} and \eqref{eq:hk:resi:eq0:u} and the notation $(\nabla \bar u)_k $ for the vectorized $\nabla \bar  u$, we see $\bar u = H^{-1}(\Div \bar \lambda + f)$ and 
	\[
	0 \in 	-\nabla \bar u + \partial g(\bar \lambda) \rightarrow  (\nabla \bar u)_k \in \mN_{\mathbb{B}_{\alpha}^k(0)}(\bar \lambda_k).
	\]
	Assumption \ref{asump:existence} means there exists a solution $(\bar u, \bar \lambda)$ for the saddle-point problem that either $ (\nabla \bar u)_k =0 $ if $\bar \lambda_k \in  \text{int} \mathbb{B}_{\alpha}^k(0)$ either  $ (\nabla \bar  u)_k \neq 0 $ if $\bar \lambda_k \in  \text{bd} \mathbb{B}_{\alpha}^k(0)$ (i.e. $|\bar \lambda_k|=\alpha$). It is not strict and coincides with the  optimality conditions in \eqref{eq:hk:resi:neq0} and \eqref{eq:hk:resi:eq0}.
\end{remark}

Henceforth, we denote $\mathcal{X}^A$ and $\mathcal{X}^I$ as the solution sets for the dual problem \eqref{eq:dual:rof} for the anisotropic TV case and the isotropic TV case respectively. With the stopping criterion \eqref{stop:a}, we have the following global and local convergence.
\begin{theorem}\label{thm:ani:KKT}
For the anisotropic TV regularized problem \eqref{eq:ROF}, denote the iteration sequence $(u^k, p^k,\lambda^k)$ generated by ALM-PDP, ALM-PDD, ALM-PT or ALM-PP with stopping criteria \eqref{stop:a}. The sequence $(u^k, p^k,\lambda^k)$  is bounded and  convergences to $(u^*, p^*, \lambda^*)$ globally. If $T_d: =\partial d$  satisfies Assumption \ref{asump:existence} and is thus metrically subregular for the origin with modulus $\kappa_d$ through Assumption \ref{asump:existence},  with the additional stopping criteria \eqref{stop:b1},  the sequence $\{\lambda^k\}$ converges to  $\lambda^* \in \mathcal{X}^A$. For  sufficiently large $k$, we have the following local linear convergence
\begin{equation}\label{eq:convergence:rate:dual:ani}
\emph{dist}(\lambda^{k+1}, \mathcal{X}^A) \leq \theta_k \emph{dist}(\lambda^k, \mathcal{X}^A),
\end{equation}
where
 \[
 \theta_k = [\kappa_d(\kappa_d^2 + \sigma_k^2)^{-1/2} + \delta_k](1-\delta_k)^{-1}, \ \emph{as} \ k\rightarrow \infty, \  \theta_k \rightarrow  \theta_{\infty} = \kappa_d(\kappa_d^2+ \sigma_{\infty}^2)^{-1/2} < 1.
 \]
If in addition that $T_l$ is metrically subregular at $(u^*, p^*, \lambda^*)$ for the origin with modulus $\kappa_l$ and the stopping criteria \eqref{stop:b2} is employed, then for sufficiently large $k$, we have
\begin{equation}\label{eq:convergence:rate:up:ani}
\|(u^{k+1}, p^{k+1}) - (u^k, p^k)\| \leq \theta_k'\|\lambda^{k+1}-\lambda^k\|,
\end{equation}
where $\theta_k'=\kappa_l(1+\delta_k')/\sigma_k$ with $\displaystyle{\lim_{k\rightarrow \infty}\theta_k' = \kappa_l/\sigma_{\infty}}$.
\end{theorem}
\begin{proof}
 Since $X$ is finite dimensional reflexive space and the primal function \eqref{eq:ROF} is l.s.c. proper convex functional and strongly convex, hence coercive. Thus the existence of the solution can be guaranteed \cite{KK} (Theorem 4.25). Furthermore, since $\dom D = X$ and $\dom \|\cdot \|_1 = Y$, by Fenchel-Rockafellar theory \cite{KK} (Chapter 4.3) (or Theorem 5.7 of \cite{Cla}), the solution to the dual problem \eqref{eq:dual:rof} is not empty and 
 \[
 \inf_{u \in X} D(u) +  \alpha \|\nabla u\|_{1} = \sup_{\lambda \in Y} -d(\lambda).
 \]
   By \cite{Roc2} (Theorem 4) (or Theorem 1 of \cite{Roc1} where the augmented Lagrangian method  essentially comes from proximal point method applying to the dual problem $\partial d$), with criterion \eqref{stop:a}, we get the boundedness of $\{\lambda^k\}$. The uniqueness of $(u^*,p^*)$ follows from the strongly convexity of $F(u)$ and the $p^*=\nabla u^*$ which is one of the KKT conditions. The boundedness of $(u^k,p^k)$ and convergence of $(u^k, p^k,  \lambda^k)$ then follows by \cite{Roc2} (Theorem 4). 
   
   The local convergence rate  \eqref{eq:convergence:rate:dual:ani} with metrical subregularity of $T_g$ and the stopping criteria  \eqref{stop:a}  \eqref{stop:b1} can be obtained from \cite{Roc2} (Theorem 5) (or Theorem 2 of \cite{Roc1}). Now we turn to the local convergence rate of $(u^k,p^k)$. By the metrical subregularity of $T_l$, for sufficiently large $k$, we have
   \[
   \| (u^{k+1}, p^{k+1}) - (u^*, p^*))\| + \text{dist} (\lambda^k, \mathcal{X}^A) \leq \kappa_l \text{dist} (0, T_l(u^{k+1}, p^{k+1}, \lambda^{k+1})).
   \]
   Together with the stopping criteria \eqref{stop:b2} and \cite{Roc2} (Theorem 5 and Corollary with formula (4.21)), we arrive at 
   \begin{align*}
   \| (u^{k+1}, p^{k+1}) - (u^*, p^*))\| &\leq \kappa_l \sqrt{{\delta_k'^2}{\sigma_k^{-2}} \|\lambda^{k+1} -\lambda^k\|^2 + {\sigma_k^{-2}} \|\lambda^{k+1} -\lambda^k\|^2 } \\
   & = \kappa_l\sqrt{\delta_k'^2+1}\sigma_k^{-1}\|\lambda^{k+1} -\lambda^k\| \leq \theta_k' \|\lambda^{k+1} -\lambda^k\|,
   \end{align*}
   which leads to \eqref{eq:convergence:rate:up:ani}.
\end{proof}
For the isotropic case, we can get similar results with metric subregularity of the dual problem.
\begin{theorem}
		For the isotropic TV regularized problem \eqref{eq:ROF}, denote the iteration sequence $(u^k, p^k,\lambda^k)$ generated by ALM-PDP, ALM-PDD, or ALM-PT with stopping criteria \eqref{stop:a}.
		Then the sequence $(u^k, p^k,\lambda^k)$  is bounded and convergences to $(u^*, p^*, \lambda^*)$ globally. If $T_d:=\partial d$ satisfies Assumption \ref{asump:existence} and is thus metrically subregular  for the origin with modulus $\kappa_d$ and with the additional stopping criteria \eqref{stop:b1}, then the sequence $\{\lambda^k\}$ converges to  $\lambda^* \in \mathcal{X}^I$ and for arbitrary sufficiently large $k$, 
		\begin{equation}
		\emph{dist}(\lambda^{k+1}, \mathcal{X}^I) \leq \theta_k \emph{dist}(\lambda^k, \mathcal{X}^I),
		\end{equation}
		where
		\[
		\theta_k = [\kappa_d(\kappa_d^2 + \sigma_k^2)^{-1/2} + \delta_k](1-\delta_k)^{-1}, \ \emph{as} \ k\rightarrow \infty, \  \theta_k \rightarrow  \theta_{\infty} = \kappa_d(\kappa_d^2+ \sigma_{\infty}^2)^{-1/2} < 1.
		\]
	\end{theorem}

 \section{Numerical Experiments}\label{sec:numer}
 In the numerics, we will first focus on the ROF denoising model for testing all the proposed algorithms, i.e., $K=I$, $\mu=0$ with $H=I$. We employ the standard finite difference discretization of the discrete gradient $\nabla$  and divergence operator $\Div$ \cite{CP}, which satisfies \eqref{eq:gradient:div:adjoint} and are convenient for operator actions based implementation. The initial value of $u^0$ is chosen as the noisy or degraded image $z$ and the dual variable is initialized zero value for all compared algorithms.    We employ the scaled sum of the residuals of $u$ and  $\lambda$ as our stopping criterion for all compared algorithms,
 \begin{equation}\label{eq:stopping:criterion}
 \text{Err}(u^{k+1}, \lambda^{k+1}): = [\text{res}(u)(u^{k+1},\lambda^{k+1}) + \text{res}(\lambda)(u^{k+1},\lambda^{k+1})]/\|f\|_F,
 \end{equation} 
where $\|\cdot\|_F$ denotes the Frobenius norm henceforth. The residual of $u$ is
\begin{equation}
\text{res}(u)(u^{k+1},\lambda^{k+1}):  = \|u^{k+1} - f + \Div \lambda^{k+1}\|_F,
\end{equation}
 and the residual of $\lambda$ is defined by
\begin{equation}
\text{res}(\lambda)(u^{k+1},\lambda^{k+1}) := \|\lambda^{k+1} - \mathcal{P}_{\alpha}(\lambda^{k+1}+ c_0 \nabla u^{k+1})\|_F.
\end{equation}
The residual $\text{res}(u)$ is the primal error of the optimality conditions in \eqref{eq:opti:primaldual} with Frobenius norm  and $\text{res}(\lambda)$ is used the measure the dual error in \eqref{eq:opti:primaldual}. Thus $ \text{Err}(u^{k+1}, \lambda^{k+1})$ in \eqref{eq:stopping:criterion} is employed to measure the primal and dual errors of the original primal model \eqref{eq:ROF}.
	
Besides the stopping criterion \eqref{eq:stopping:criterion}, the following residual originated in \cite{KK1} is also useful
\begin{equation}
\text{Res1}(u^{k+1},\lambda^{k+1}):= \|\mI_{\{\norm[\infty]{ \lambda} \leq \alpha\}}( \lambda^{k+1}) + \alpha | \nabla u^{k+1}| - \langle  \lambda^{k+1}, \nabla u^{k+1}\rangle\|_F.
\end{equation}
 We will also  consider the following residual for measuring the optimality conditions  \eqref{eq:hk:resi:neq0} and \eqref{eq:hk:resi:eq0} originated from \cite{HS}  
\begin{equation}
\text{Res2}(u^{k+1},\lambda^{k+1})=\|\alpha \nabla u^{k+1}-|\nabla  u^{k+1}| \lambda^{k+1} \|_F. 
\end{equation}
The following primal-dual gap for the ROF model as in \cite{CP} is also a useful residual
\begin{equation*}
\gap(u^{k+1},\lambda^{k+1}) = \frac{\norm[2]{u^{k+1}-f}^2}{2}
+ \alpha \norm[1]{\grad u^{k+1}} + \frac{\norm[2]{\Div \lambda^{k+1} + f}^2}{2}
- \frac{\norm[2]{f}^2}{2} + \mI_{\{\norm[\infty]{\lambda} \leq \alpha\}}(\lambda^{k+1}).
\end{equation*}
We use the following normalized primal-dual gap \cite{CP}
\begin{equation}  \label{eq:l2-tv-gap}
\text{gap}(u^{k+1},\lambda^{k+1}) :  =\gap(u^{k+1},\lambda^{k+1}) /NM, \quad \text{with} \ \  NM = N*M, \ \  u^k \in \mathbb{R}^{N\times M}.
\end{equation}

Let's now turn to the stopping criterion for linear iterative solvers including  CG  (conjugate gradient) and BiCGSTAB ( biconjugate gradient stabilized method) for each linear system for the Newton update in Algorithm \ref{alm:SSN_PDP}, \ref{alm:SSN_PDD}, \ref{alm:SSN_PT}.  For 
Algorithm \ref{alm:SSN_PT}, we use CG  due to the symmetric linear system \eqref{eq:ssn:PT:ani}, and \eqref{eq:ssn:PT:iso}. For any other linear system, e.g.,  \eqref{eq:calculate:u:first}, \eqref{eq:ssn:pdd:h}, we use BiCGSTAB (see Figure 9.1 of \cite{VAN}), which is very efficient for nonsymmetric linear system. The following stopping criterion is employed for solving linear systems to get  Newton updates with BiCGSTAB or CG \cite{HS}, 
\begin{equation}\label{eq:stop:bicg}
\text{tol}_{k+1}: =.1\min\left\{ \left(\frac{\text{res}_k}{\text{res}_0}\right)^{1.5},  \  \frac{\text{res}_k}{\text{res}_0}  \right\},
\end{equation}
which can help catch the superlinear convergence of semismooth Newton we employ. The $\text{res}_k$ in \eqref{eq:stop:bicg} denotes the relative residual of the Newton linear system after the $k$-th BiCGSTAB or CG iteration while $\text{res}_0$ denotes the original residual before BiCGSTAB or CG iterations. 

Now, we turn to the most important stopping criterion \eqref{stop:a}, \eqref{stop:b1}, \eqref{stop:b2} of each ALM iteration for determining how many Newton iterations are needed when solving the corresponding nonlinear systems \eqref{eq:opti:u:lambda} or \eqref{eq:u:alm:suntoh1}. The criterion \eqref{stop:b1} is not practical. New stopping criterion of ALM for cone programming can be found in \cite{CST}.
We found the following empirical stopping criterion for each ALM iteration is efficient numerically.
For the primal-dual semismooth Newton with auxiliary variable, we empirical employ
\begin{equation}\label{eq:stop:alm:ssn:dual}
\text{res-alm}_{SSNPD}^l = \|- \sigma_k \nabla u^l -  \lambda^k + \max \big(1.0, \dfrac{|\lambda^k + \sigma_k \nabla u^l|}{\alpha} \big) h^l\|_F,
\end{equation}
where $h^l$ is computed by Algorithm \ref{alm:SSN_PDP} or \ref{alm:SSN_PDD} before the projection to the feasible set. 
For the semismooth Newton method involving the soft thresholding operator (SSNPT), we employ
\begin{equation}\label{eq:alm:stop:ssnpt}
\text{res-alm}_{SSNPT}^l: = \|u^l - f + \nabla^* \lambda^k + \sigma_k \nabla^*\nabla u^l -\sigma_k \nabla^*(I + \frac{1}{\sigma_k} \partial \alpha \|\cdot\|_{1})^{-1}(\frac{\lambda^k + \sigma_k \nabla u^l}{\sigma_k})\|_F.
\end{equation}
Here we use $\text{res-alm}_{SSNPD}^l $ without $\text{res}(u)$ compared to \eqref{stop:b2} since we found $\text{res}(u)$ is usually much smaller compared to  \eqref{eq:stop:alm:ssn:dual} in numerics. We employ the following stopping criterion
\begin{equation}\label{eq:stop:alm}
\text{res-alm}_{SSNPD}^l \quad \text{or} \quad  \text{res-alm}_{SSNPT}^l \leq {\delta_k}/{\sigma_k},
\end{equation}
where $\delta_k$ is a small parameter which can be chosen as fixed constants $10^{-2}$, $10^{-4}$ and so on in our numerical tests. We emphasis that divided by $\sigma_k$ is of critical importance for the convergence of ALM, which is also required by the stopping criterion \eqref{stop:a}, \eqref{stop:b1}, \eqref{stop:b2}. 
\begin{table}
	\centering 
	\begin{tabular}{l@{}l@{}rl@{}rl@{}rl@{}rl@{}rl@{}rl@{}rl@{}l@{}r} 
		\toprule
		& \multicolumn{6}{c}{Anisotropic TV: $\alpha = 0.1$ }
		& \multicolumn{9}{c}{Lena: $256 \times 256$}\\
		\midrule
		& \multicolumn{2}{c}{$n(t)$}
		& \multicolumn{2}{c}{$\text{res}(u)$}
		& \multicolumn{2}{c}{$\text{res}(\lambda)$}
		& \multicolumn{2}{c}{Res1}
		& \multicolumn{2}{c}{Res2}
		& \multicolumn{2}{c}{Gap}
		& \multicolumn{2}{c}{$\text{PSNR}$}
		& \multicolumn{2}{c}{Err}\\
		\midrule
		ALM-PDP &&4(5.42s)  && 2.46e-5 &&6.96e-3  &&4.98e-5  &&3.53e-5 && 5.89e-8&&19.25&&1e-4\\
		ALM-PDD && 4(23.98s)  && 1.68e-13 && 6.96e-3&& 4.98e-5 &&3.53e-5 &&5.89e-8&&19.25&&1e-4\\
		ALM-PT && 7(13.80s)  && 2.26e-4 && 7.57e-3 && 5.49e-5 &&3.89e-5&&5.88e-8&&19.25&&1e-4\\
		\midrule
		ALG2 &&1032(6.71s)  && 1.30e-2 &&6.92e-4 &&7.86e-6 &&5.56e-6 &&1.00e-8&&   19.25&&1e-4\\
	    \midrule 
        \midrule 
		ALM-PDP &&6(11.45s)  && 3.66e-5 &&9.12e-6  &&6.91e-8 &&4.93e-8&&2.58e-11&&19.25&&1e-6\\
		ALM-PDD && 7(56.49s)  && 3.66e-5 && 7.78e-5&& 1.55e-6 &&1.11e-6 &&7.85e-11&&19.25&&1e-6\\
		ALM-PT &&9(18.29s)  && 8.66e-6 && 3.73e-5 && 2.87e-7 &&2.04e-7&&1.15e-10&&19.25&&1e-6\\
		
		\midrule
		ALG2 &&14582(90.28s)  && 1.37e-4 &&2.26e-7  &&2.46e-9 &&1.83e-9 &&4.93e-13&&19.25&&1e-6\\
		\midrule		
		\midrule		
		ALM-PDP &&9(18.83s)  && 1.06e-6 &&3.74e-10 &&3.03e-12 &&2.16e-12 &&4.02e-16&&19.25&&1e-8\\
		ALM-PDD && 7(69.76s)  && 2.80e-7 && 1.07e-7&& 8.62e-10 &&6.22e-10 &&2.02e-13&&19.25&&1e-8\\
		ALM-PT && 11(28.26s)  && 6.41e-7 && 3.12e-9&&5.45e-12 &&3.94e-12&&1.29e-15&&19.25&&1e-8\\
		\midrule
		ALG2 &&365019(2069.36s)  && 1.37e-6 &&2.88e-13  &&3.37e-15 &&2.38e-15 &&$<$eps&&19.25&&1e-8\\
		\bottomrule 
		\bottomrule 
	\end{tabular}
	\vspace*{-0.5em}
	\caption{For $n(t)$ of the first column of each algorithm, $n$ presents the outer ALM iteration number or iteration number for ALG2 for the scaled residual Err$(u^{k+1}, \lambda^{k+1})$ less than the stopping value, $t$ denoting the CPU time. The notation ``$<$eps" denotes the corresponding quality less than the machine precision, i.e.,  the ``eps" in Matlab. }
	\label{tab:rof:lena:ani}
\end{table}

\begin{table}
	\centering 
	\begin{tabular}{l@{}l@{}rl@{}rl@{}rl@{}rl@{}rl@{}rl@{}rl@{}l@{}r} 
		\toprule
		& \multicolumn{6}{c}{Isotropic TV: $\alpha = 0.1$ }
		& \multicolumn{8}{c}{Cameraman: $256 \times 256$}\\
		\midrule
		& \multicolumn{2}{c}{$n(t)$}
		& \multicolumn{2}{c}{$\text{res}(u)$}
		& \multicolumn{2}{c}{$\text{res}(\lambda)$}
		& \multicolumn{2}{c}{Res1}
		& \multicolumn{2}{c}{Res2}
		& \multicolumn{2}{c}{Gap}
		& \multicolumn{2}{c}{$\text{PSNR}$}
		& \multicolumn{2}{c}{Err}\\
		\midrule
		ALM-PDP &&3(4.75s)  && 2.32e-3 &&1.14e-2  &&5.61e-5  &&9.36e-5 && 9.70e-8&&19.18&&1e-4\\
		ALM-PDD && 5(57.29s)  && 8.61e-3 && 2.61e-3&& 1.37e-5 &&2.19e-5 &&2.01e-8&&19.18&&1e-4\\
		ALM-PT && 8(174.07s)  && 5.23e-5 && 4.30e-3 && 2.31e-5 &&3.64e-5&&3.34e-8&&19.18&&1e-4\\
		\midrule
		ALG2 &&479(2.65s)  && 1.59e-2 &&6.03e-6 &&7.38e-6 &&7.01e-4 &&8.69e-9&&   19.18&&1e-4\\
		\midrule 
		\midrule 
		ALM-PDP &&7(20.35s)  && 6.75e-5 &&8.08e-5  &&4.46e-7 &&6.96e-7&&4.01e-10&&19.18&&1e-6 \\
			ALM-PDD && 7(191.50s)  && 3.98e-5 && 4.94e-5&& 2.73e-7 &&4.24e-7 &&2.28e-10&&19.18&&1e-6\\ 
		ALM-PT &&11(787.01s)  && 1.89e-5 && 6.97e-5 && 3.91e-7 &&6.02e-7&&3.29e-10&&19.18&&1e-6\\
		\midrule
		ALG2 &&10808(55.16s)  && 1.60e-4 &&1.60e-9 &&2.46e-9 &&2.63e-7 &&6.56e-13&&19.18&&1e-6\\
		\midrule		
		\midrule		
		ALM-PDP &&10(243.57s)  && 1.07e-6 &&4.22e-7 &&2.43e-9 &&3.82e-9 &&1.04e-12&&19.18&&1e-8\\
		ALM-PDD && --- && --- && ---&&--- &&--- &&---&&---&&1e-8\\
		ALM-PT && ---  && --- && ---&&--- &&---&&---&&---&&1e-8\\
		\midrule
		ALG2 &&262811(1348.55s)  && 1.60e-6 &&5.01e-13  &&7.47e-13 &&8.11e-11 &&$<$eps&&19.18&&1e-8\\
		\bottomrule 
		\bottomrule 
	\end{tabular}
	\vspace*{-0.5em}
	\caption{For $n(t)$ of the first column of each algorithm, $n$ presents the outer ALM iteration number or iteration number for ALG2 for the scaled residual Err$(u^{k+1}, \lambda^{k+1})$ less than the stopping value, $t$ denoting the CPU time. ``---" denotes the iteration time more than 2000 seconds. The notation ``$<$eps" is the same as in Table \ref{tab:rof:lena:ani}.}
	\label{tab:rof:cameraman:iso}
\end{table}

\begin{table}
	\centering 
	\begin{tabular}{l@{}l@{}rl@{}rl@{}rl@{}rl@{}rl@{}rl@{}rl@{}l@{}r} 
		\toprule
		& \multicolumn{6}{c}{Anisotropic TV: $\alpha = 0.1$ }
		& \multicolumn{9}{c}{Sails: $768 \times 512$}\\
		\midrule
		& \multicolumn{2}{c}{$n(t)$}
		& \multicolumn{2}{c}{$\text{res}(u)$}
		& \multicolumn{2}{c}{$\text{res}(\lambda)$}
		& \multicolumn{2}{c}{Res1}
		& \multicolumn{2}{c}{Res2}
		& \multicolumn{2}{c}{Gap}
		& \multicolumn{2}{c}{$\text{PSNR}$}
		& \multicolumn{2}{c}{Err}\\
		\midrule
		ALM-PDP &&6(25.50s)&& 1.56e-2 &&1.15e-3  &&7.31e-5 &&2.28e-5&& 1.61e-5&&18.90&&1e-4\\
		ALM-PDD && 4(159.38)  && 7.60e-3 && 1.97e-2&&2.73e-4&&2.28e-4&&4.76e-8&&18.90&&1e-4\\
	    ALM-PT && 7(101.59s)  && 2.72e-4 && 1.53e-2&& 1.15e-4 &&8.16e-5&&4.55e-8  &&18.90&&1e-4\\
		\midrule
		ALG2 &&933(28.23s)   && 3.35e-2 &&1.83e-3 &&2.10e-5 &&1.49e-5 &&1.08e-8&&18.90&&1e-4\\
		\midrule 
		\midrule 
		ALM-PDP &&7(52.52s)  && 5.68e-5 &&2.25e-6 &&2.03e-8 &&1.45e-8&&9.42e-13&&18.90&&1e-6\\
		ALM-PDD && 6(377.60s)&& 2.18e-5 && 1.13e-5&& 9.03e-8 &&6.43e-8 &&1.25e-11&&18.90&&1e-6\\
		ALM-PT &&8(337.23s)  && 2.20e-5&& 5.58e-5 && 4.42e-7 && 3.17e-7 &&6.11e-11&&18.90&&1e-6\\
		\midrule
		ALG2 &&14842(476.20s)  && 3.53e-4 &&3.40e-7 &&3.94e-9 &&2.99e-9 &&3.09e-13&&18.90&&1e-6\\
		\midrule		
		\midrule		
		ALM-PDP &&7(54.51s) && 5.65e-7 &&2.25e-6 &&2.03e-8 &&1.45e-8 &&9.38e-13&&18.90&&1e-8\\
			ALM-PDD &&14(122.23s) && 2.68e-6 &&3.69e-12 &&3.44e-14 &&2.47e-14 &&$<$eps&&18.90&&1e-8\\
		ALM-PT && 10(363.27s) && 3.30e-6 &&1.71e-8&&1.61e-10 &&1.14e-10&&1.44e-14&&18.90&&1e-8\\ 
		\midrule
		ALG2 &&371060(11339.88s)  && 3.54e-6 &&6.98e-13  &&8.20e-15 &&5.80e-15 &&$<$eps&&18.90&&1e-8\\
		\bottomrule 
		\bottomrule 
	\end{tabular}
	\vspace*{-0.5em}
	\caption{For $n(t)$ of the first column of each algorithm, $n$ presents the outer ALM iteration number or iteration number for ALG2 for the scaled residual Err$(u^{k+1}, \lambda^{k+1})$ less than the stopping value, $t$ denoting the CPU time. The notation ``$<$eps" is the same as in Table \ref{tab:rof:lena:ani}.}
	\label{tab:rof:sails:ani}
\end{table}

\begin{table}
	\centering 
	\begin{tabular}{l@{}l@{}rl@{}rl@{}rl@{}rl@{}rl@{}rl@{}rl@{}l@{}r} 
		\toprule
		& \multicolumn{6}{c}{Isotropic TV: $\alpha = 0.1$ }
		& \multicolumn{9}{c}{Monarch: $768 \times 512$}\\
		\midrule
		& \multicolumn{2}{c}{$n(t)$}
		& \multicolumn{2}{c}{$\text{res}(u)$}
		& \multicolumn{2}{c}{$\text{res}(\lambda)$}
		& \multicolumn{2}{c}{Res1}
		& \multicolumn{2}{c}{Res2}
		& \multicolumn{2}{c}{Gap}
		& \multicolumn{2}{c}{$\text{PSNR}$}
		& \multicolumn{2}{c}{Err}\\
		\midrule
		ALM-PDP &&3(25.12)  && 1.12e-2 &&1.43e-2  &&7.31e-5 &&1.20e-4&& 4.29e-8&&29.89&&1e-4\\
		ALM-PDD && 5(370.53s)  && 9.89e-3 && 5.82e-3&& 3.01e-5&&4.89e-5&&1.54e-8&&29.89&&1e-4\\
     	ALM-PT &&8(2009.60s)  && 8.34e-5 && 9.73e-3 && 5.13e-5 &&8.19e-5&&2.66e-8&&29.89&&1e-4\\
		\midrule
		ALG2 &&525(15.77s)  && 2.86e-2 &&7.69e-6 &&1.02e-5 &&1.07e-3 &&4.76e-9&&29.89&&1e-4\\
		\midrule 
		\midrule 
	    ALM-PDP &&7(145.27s)  && 7.98e-5 &&1.80e-4  &&9.79e-7 &&1.54e-6&&3.15e-10&&29.89&&1e-6\\
		ALM-PDD && 7(1379.33s)  && 1.47e-7 && 1.09e-4&& 5.96e-7 &&9.33e-7 &&1.78e-10&&29.89&&1e-6\\
		ALM-PT &&---  && --- && --- && --- &&---&&---&&---&&1e-6\\
		\midrule
		ALG2 &&12049(387.21s)  && 2.87e-4 &&2.61e-9 &&3.79e-9 &&4.27e-7 &&3.63e-13&&29.89&&1e-6\\
		\midrule		
		\midrule		
		ALM-PDP &&10(1549.91s)  && 1.79e-6 &&1.06e-6 &&5.90e-9 &&9.05e-9 &&8.94e-13&&29.89&&1e-8\\
		ALM-PDD && --- && --- && ---&&--- &&--- &&---&&---&&1e-8\\
		ALM-PT && ---  && --- && ---&&--- &&---&&---&&---&&1e-8\\
		\midrule
		ALG2 &&292442(8590.28s)  && 2.87e-6 &&1.56e-12  &&2.11e-12 &&2.33e-10 &&$<$eps&&29.89&&1e-8\\
		\bottomrule 
		\bottomrule 
	\end{tabular}
	\vspace*{-0.5em}
	\caption{For $n(t)$ of the first column of each algorithm, $n$ presents the outer ALM iteration number or iteration number for ALG2 for the scaled residual Err$(u^{k+1}, \lambda^{k+1})$ less than the stopping value, $t$ denoting the CPU time. For ALM based algorithms, ``---" denotes the cases when the iteration time more than 3000 seconds. We keep ALG2 for comparison. }
	\label{tab:rof:butterfly:iso}
\end{table}

\begin{table}
	\centering 
	\begin{tabular}{l@{}l@{}rl@{}rl@{}rl@{}rl@{}rl@{}rl@{}rl@{}l@{}r} 
		\toprule
		& \multicolumn{9}{c}{Isotropic TV deblurring: $\alpha = 0.1$ }
		& \multicolumn{5}{c}{Pepper: $512 \times 512$}\\
				\midrule
					\midrule
		& \multicolumn{9}{c}{Err$=10^{-5}$, $\mu = 10^{-6}$ }
		& \multicolumn{5}{c}{Err$=10^{-5}$, $\mu = 10^{-9}$}\\
		\midrule
		& \multicolumn{2}{c}{$n(t)$}
		& \multicolumn{2}{c}{$\text{res}(u)$}
		& \multicolumn{2}{c}{$\text{res}(\lambda)$}
		& \multicolumn{2}{c}{$\text{PSNR}$}
	    & \multicolumn{2}{c}{$n(t)$}
		& \multicolumn{2}{c}{$\text{res}(u)$}
		& \multicolumn{2}{c}{$\text{res}(\lambda)$}
		& \multicolumn{2}{c}{$\text{PSNR}$}\\
		\midrule
		ALM-PDP && 3(126.08s)&&1.06e-4&&1.08e-3&&28.48 &&3(117.37s) && 4.36e-4 &&1.07e-3&&28.48\\
		\midrule
		ALG2 &&4775(178.18s) &&2.02e-3&&5.48e-4&&28.36 &&4762(170.65s) &&2.02e-3&&   5.72e-4&&28.36\\
		\midrule 
		\midrule
		& \multicolumn{9}{c}{Err$=10^{-6}$, $\mu = 10^{-6}$ }
& \multicolumn{5}{c}{Err$=10^{-6}$, $\mu = 10^{-9}$}\\
\midrule
		& \multicolumn{2}{c}{$n(t)$}
& \multicolumn{2}{c}{$\text{res}(u)$}
& \multicolumn{2}{c}{$\text{res}(\lambda)$}
& \multicolumn{2}{c}{$\text{PSNR}$}
& \multicolumn{2}{c}{$n(t)$}
& \multicolumn{2}{c}{$\text{res}(u)$}
& \multicolumn{2}{c}{$\text{res}(\lambda)$}
& \multicolumn{2}{c}{$\text{PSNR}$}\\
\midrule
		ALM-PDP &&4(196.48s)  && 3.07e-5 &&9.69e-5  &&28.48 &&4(186.79s)&&7.08e-5&&1.70e-4&&28.48 \\
		\midrule
		ALG2 &&11017(397.71s) && 2.02e-4 &&5.58e-5 &&28.47 &&10911(395.16s) &&2.02e-4&&5.54e-5&&28.47\\
		\bottomrule 
	\end{tabular}
	\vspace*{-0.5em}
	\caption{For $n(t)$ of the first column of each algorithm, $n$ presents the outer ALM iteration number or iteration number for ALG2 for the scaled residual Err$(u^{k+1}, \lambda^{k+1})$ less than the stopping value, $t$ denoting the CPU time. ``---" denotes the iteration time more than 2000 seconds. The notation ``$<$eps" is the same as in Table .}
	\label{tab:deblur:iso}
\end{table}

For numerical comparisons, we mainly choose the accelerated primal-dual algorithm ALG2 \cite{CP} with asymptotic convergence rate $\mathcal{O}(1/k^2)$, which is very efficient, robust and standard algorithms for imaging problems. We follow the same parameters for ALG2 as in \cite{CP} and the corresponding software for the ROF model. Here we use the stopping criterion \eqref{eq:stopping:criterion} instead of the normalized primal-dual gap \eqref{eq:l2-tv-gap} as in \cite{CP}.   
The test images Lena and Cameraman are with size $256\times 256$ and images Monarch and Sails are with size $768 \times 512$. The original, noisy and denoised images can be seen in Figure \ref{lena:cameraman:denoise} and \ref{monarch:sails:denoise}. We acknowledge that the residuals $\text{Res1}(u^{k+1},\lambda^{k+1})$ and $\text{Gap}(u^{k+1},\lambda^{k+1})$ can be unbounded without projection of $\lambda^{k+1}$ to the feasible set as in ALM-PT. We leave them in the tables for references. 

From Table \ref{tab:rof:lena:ani} and  Table \ref{tab:rof:sails:ani}, it can be seen that the proposed ALM-PDP, ALM-PDD and ALM-PT are very efficient and competitive for the anisotropic ROF model and are very robust for different sizes of images. From Table \ref{tab:rof:cameraman:iso} and Table \ref{tab:rof:butterfly:iso}, it can be seen that the proposed ALM-PDP are still very efficient and competitive for the isotropic cases.  
 The performance of ALG2 seems to hold nearly the same efficiency for the isotropic and anisotropic cases. Although the residual $\text{res}(\lambda)$  for smaller  $\text{Err}(u^{k+1}, \lambda^{k+1})$ are much smaller for ALG2 than the proposed semismooth Newton based ALM methods, however,  the residual $\text{res}(u)$ of ALG2 is more higher than the proposed ALM methods. It seems ALG2 is more efficient for decreasing the dual error $\text{res}(\lambda)$.

\begin{table}
	\centering 
	\begin{tabular}{l*{14}{c}r}
		\midrule
		& \multicolumn{1}{c}{$k=1$}
		& \multicolumn{1}{c}{$k=2$}
		& \multicolumn{1}{c}{$k=3$}
		& \multicolumn{1}{c}{$k=4$}
		& \multicolumn{1}{c}{$k=5$}
		& \multicolumn{1}{c}{$k=6$}
		\\
		\midrule
		res($u$)  &3.50e-6     &5.30e-7  &1.53e-5&6.14e-8  &3.29e-8 &1.94e-7 \\
		res($\lambda$) & 9.70   &1.09  &9.99e-2 &6.96e-3  &3.07e-4 &9.13e-6\\
		Gap   &1.61e-4     &1.41e-5  &1.14e-6&5.89e-8  &1.58e-9 &2.58e-11 \\
		$N_{SSN}$  &10     &9  &8&11  &9 &7 \\
		$N_{ABCG}$ & 20   &29  &51 &56  &102 &116\\
		
		\bottomrule 
	\end{tabular}	
	\vspace*{-0.5em}
	\caption{Image Lena denoised by anisotropic TV with algorithm ALM-PDP. 	$N_{SSN}$ denotes the number of Newton iterations for each ALM iteration.  $N_{ABCG}$ denotes average number of BiCGSTAB iterations for each Newton iteration. Here $\sigma_0 = 4$, $\sigma_{k+1} = 4\sigma_k$.}
	\label{tab:lena:ssn:iter:ani}
\end{table}

%

Our proposed algorithms are  much more efficient for the anisotropic case compared to the isotropic case for ROF denoising model. Besides Theorem  \ref{thm:metric:regular:lag} and \ref{thm:metric:regular:dual:iso}, we present another comparison between the isotropic TV and anisotropic TV with algorithm ALM-PDP.  For Table \ref{tab:lena:ssn:iter:ani}, we give enough Newton iterations compared to previous tables with the stopping the criterion  $\text{res-alm}_{SSNPD}^l \leq {10^{-3}}/{\sigma_k}$ for each ALM as in \eqref{eq:stop:alm} and  $\text{tol}_{k+1}\leq 10^{-5}$ in \eqref{eq:stop:bicg} for each Newton iteration. The ALM iterations stop while Err$(u^{k+1},\lambda^{k+1})\leq 10^{-7}$ both for the anisotropic TV.

Finally, let's turn to the image deblurring problem. For the image deblurring with  $Ku = \kappa \ast u$ being the convolution with kernel $\kappa$ and additional regularization term $\frac{\mu}{2}\|\nabla u\|_{2}^2$, we mainly focus on the isotropic case as in Table \ref{tab:deblur:iso}. The original, degraded and reconstructed Pepper images with size $512\times 512$ can be found in Figure \ref{lena:cameraman:denoise}. Since the primal model \eqref{eq:ROF} is only strongly convex on the primal variable $u$ instead being strongly convex on both $u$ and $\lambda$, we thus employ ALG2 for comparison. The strongly convex parameter for $D(u)$ is a subtle issue since $K^*K$ is not invertible with the kernel.  We use the strongly convex parameter of $\frac{\mu}{2}\|\nabla u\|_{2}^2$ instead.   It can be seen that both ALM-PDP and ALG2 are very efficient for small $\mu$ case. It is interesting that both ALM-PDP and ALG2 seem more efficient for the case $\mu=10^{-9}$ than the case $\mu=10^{-6}$. 

\begin{figure}
	\begin{center}
		\subfloat[Original image: lena]
		{\includegraphics[width=3.0cm]{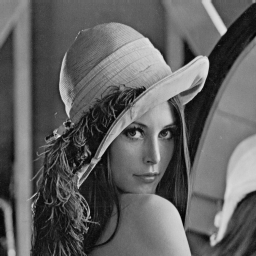}}\quad
		\subfloat[Noisy image: $10\%$ Gaussian]
		{\includegraphics[width=3.0cm]{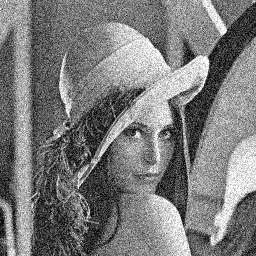}} \quad
		\subfloat[ALM-PDP($10^{-6}$)]
		{\includegraphics[width=3.0cm]{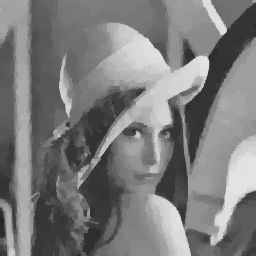}} \\
		\subfloat[Original image: Cameraman]
		{\includegraphics[width=3.0cm]{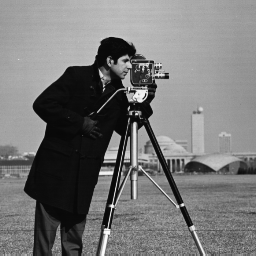}}\quad
		\subfloat[Noisy image: $10\%$ Gaussian]
		{\includegraphics[width=3.0cm]{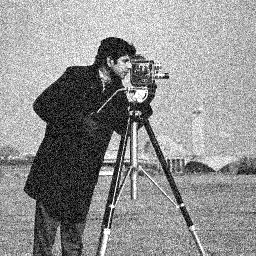}} \quad
		\subfloat[ALM-PDP($10^{-6}$)]
		{\includegraphics[width=3.0cm]{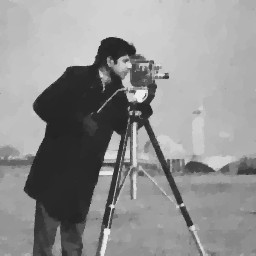}} \\
						\subfloat[Original image: Pepper]
		{\includegraphics[width=3.0cm]{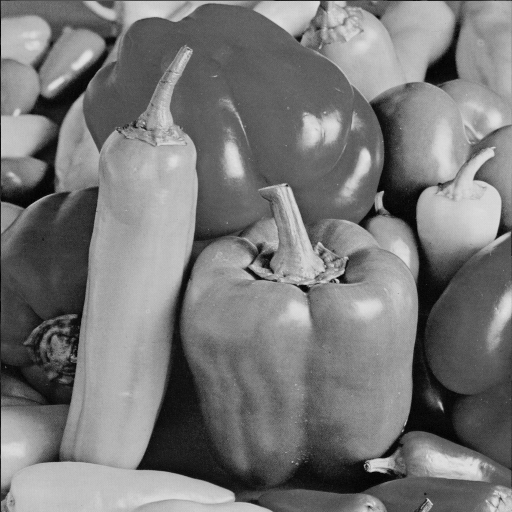}}\quad
		\subfloat[Degraded image:]
		{\includegraphics[width=3.0cm]{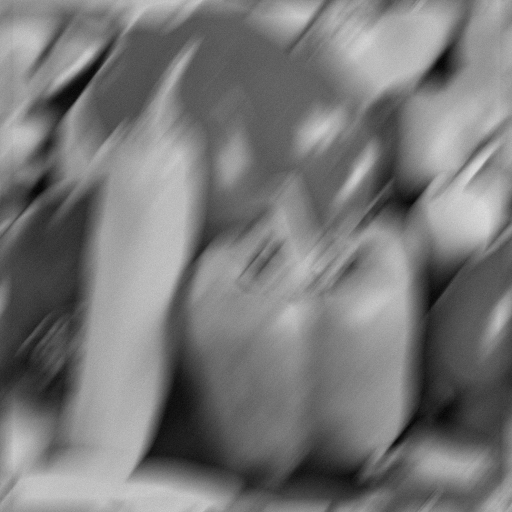}} \quad
		\subfloat[ALM-PDP($10^{-5}$)]
		{\includegraphics[width=3.0cm]{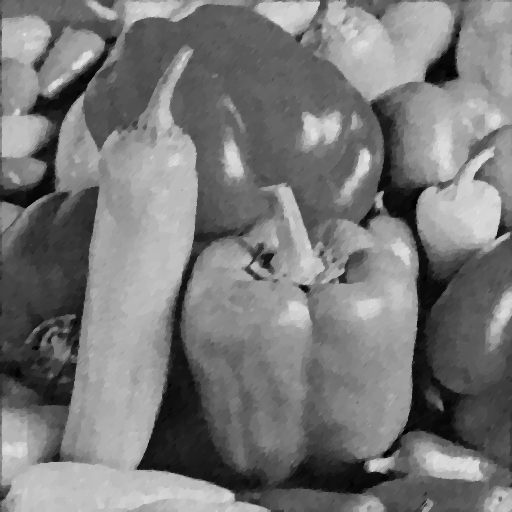}} \\
	\end{center}
	\caption{\scriptsize{Images (a) and (d) show the original $256 \times 256$ Lena and Cameraman images. Image (g) shows the original $512 \times 512$ Pepper image. (b) and (e) are  noisy versions corrupted by 10\% Gaussian noise. Image (h) is the degraded Pepper image ($\approx 40$ pixels motion blur, Gaussian noise, standard
			deviation $0.01$). (c) and (f) show the denoised images with ALM-PDP with Err$(u^{k+1}, \lambda^{k+1})<10^{-6}$. Image Lena is denoised by anisotropic TV while image Cameraman is denoised by isotropic TV. Image (i) is the deblurred Pepper image with isotropic TV by ALM-PDP with Err$(u^{k+1}, \lambda^{k+1})<10^{-5}$. }}
	\label{lena:cameraman:denoise}
\end{figure}

\begin{figure}
	\begin{center}
		\subfloat[Original image: sails]
		{\includegraphics[width=3.5cm]{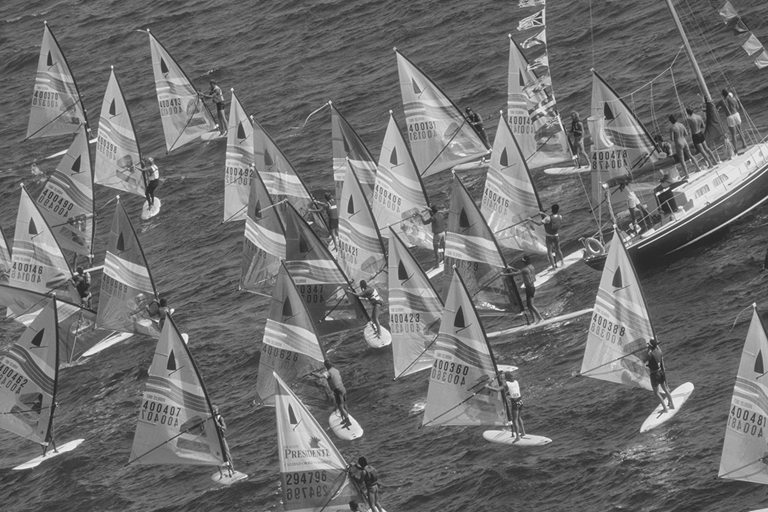}}\quad
		\subfloat[Noisy image: $10\%$ Gaussian]
		{\includegraphics[width=3.5cm]{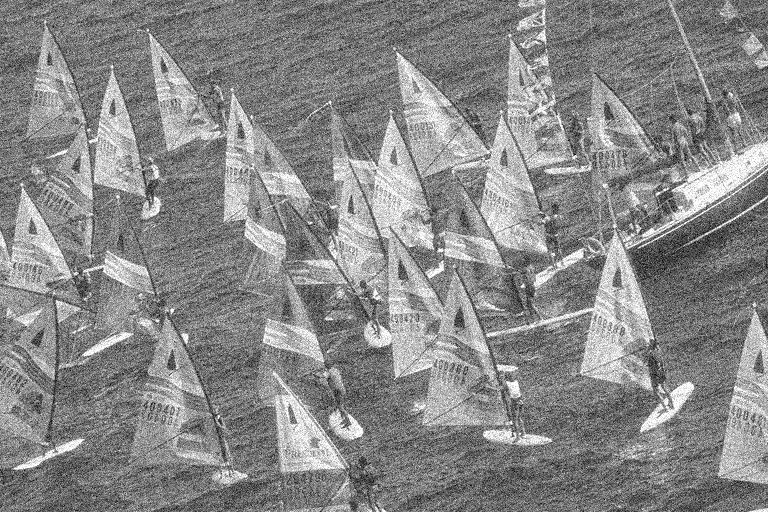}} \quad
		\subfloat[ALM-PDP($10^{-6}$)]
		{\includegraphics[width=3.5cm]{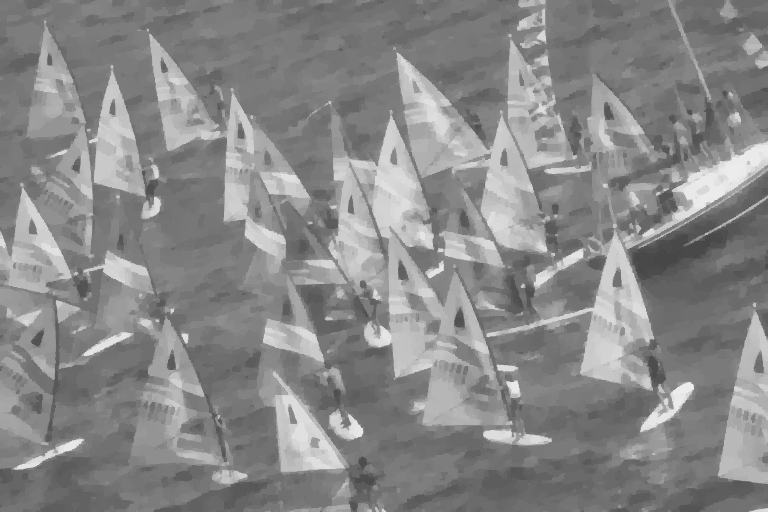}} \\
		\subfloat[Original image: Monarch]
		{\includegraphics[width=3.5cm]{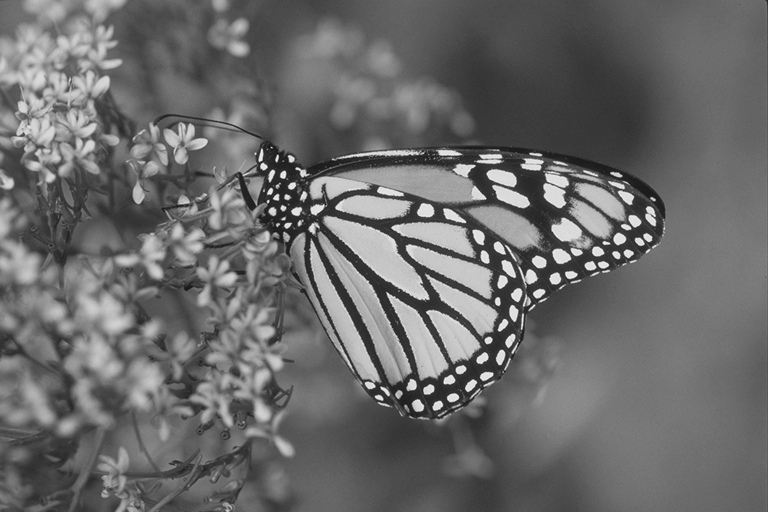}}\quad
		\subfloat[Noisy image: $10\%$ Gaussian]
		{\includegraphics[width=3.5cm]{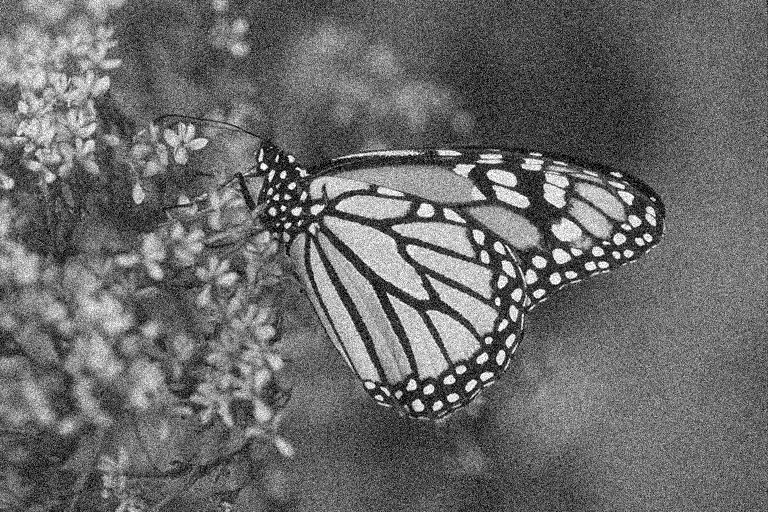}} \quad
		\subfloat[ALM-PDP($10^{-6}$)]
		{\includegraphics[width=3.5cm]{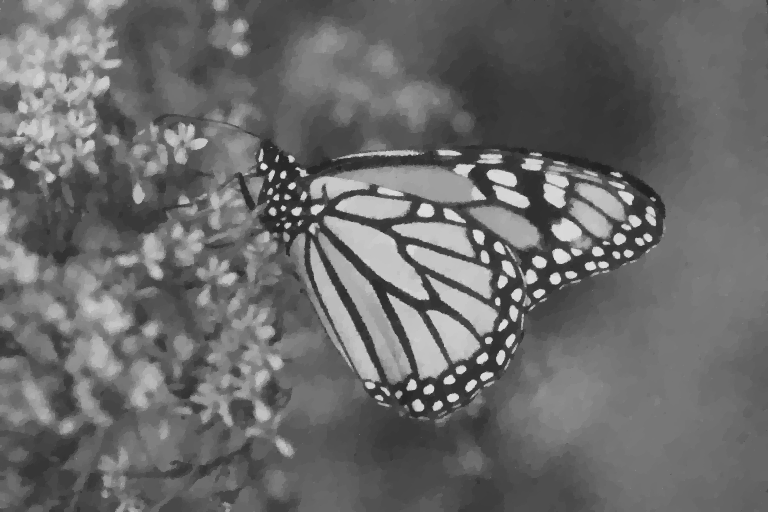}} \\
	\end{center}
	\caption{Images (a) and (d) show the original $768 \times 512$ Sails and Monarch images. (b) and (e) are  noisy versions corrupted by 10\% Gaussian noise. (c) and (f) show the denoised images with ALM-PDP with Err$(u^{k+1}, \lambda^{k+1})<10^{-6}$. Image Sails is denoised by anisotropic TV while image Monarch is denoised by isotropic TV.}
	\label{monarch:sails:denoise}
\end{figure}

\section{Discussion and Conclusions}\label{sec:conclude}
In this paper, we proposed several semismooth Newton based ALM algorithms. The proposed algorithms are very efficient and competitive especially for anisotropic TV. Global convergence and the corresponding asymptotic convergence rates are also discussed with metrical subregularites. Numerical tests show that compared with first-order algorithm, more computation efforts are deserved with ALM.  Currently, no preconditioners are employed by BiCGSTAB or CG for the linear systems solving Newton updated. Actually, preconditioners for BiCGSTAB or CG are desperately needed, especially for the isotropic cases. Additionally, the asymptotic convergence rate of the KKT residuals is also an interesting topic for future study \cite{CST}.

\noindent
{\small
	\textbf{Acknowledgements}
	The author acknowledges the support of
	NSF of China under grant No. \,11701563.	The work was originated during the author's visit to Prof. Defeng Sun of the Hong Kong Polytechnic University in October 2018. The author is very grateful to Prof. Defeng Sun
for introducing the framework on semismooth Newton based ALM developed by him and his collaborators and for his suggestions on the self-adjointness of the corresponding operators for Newton updates  where CG can be employed. The author is also very grateful to Prof. Kim-Chuan Toh, Dr. Chao Ding, Dr. Xudong Li and Dr. Xinyuan Zhao for the discussion on the semismooth Newton based ALM. The author is also very grateful to Prof. Michael Hinterm{\"u}ller for the discussion on the primal-dual semismooth Newton method during the author's visit to Weierstrass Institute for Applied Analysis and Stochastics (WIAS) supported by Alexander von Humboldt Foundation during 2017. 
}


\end{document}